\documentclass{bcp}
\usepackage{amsmath,amsthm}
\usepackage{amssymb}

\usepackage{fullpage,amsfonts,amsmath,amssymb}

\usepackage{verbatim}
\usepackage{float}

\theoremstyle{plain} \newtheorem{Thm}{Theorem}[section]
\theoremstyle{plain} \newtheorem{Cor}[Thm]{Corollary}
\theoremstyle{plain} \newtheorem{Prop}[Thm]{Proposition}
\theoremstyle{plain} \newtheorem{Lemma}[Thm]{Lemma}
\theoremstyle{plain} 
\theoremstyle{definition} \newtheorem{Def}[Thm]{Definition}
\theoremstyle{definition} \newtheorem{Rem}[Thm]{Remark}
\theoremstyle{definition} 
\theoremstyle{definition} \newtheorem{Ex}[Thm]{Example}
\theoremstyle{plain}

\numberwithin{equation}{section}
\def\frak{\mathfrak}

\renewcommand{\Re}{\mathop{\rm{Re}}}
\renewcommand{\Im}{\mathop{\rm{Im}}}

\newcommand{\sgn}{\mathop{\rm{sgn}}}

\newcommand{\id}{\mathrm{id}}

\newcommand{\Tr}{\mathop{\rm{Tr}}}
\newcommand{\diag}{\mathop{\rm{diag}}}
\newcommand{\supp}{\mathop{\rm{supp}}}
\newcommand{\Ad}{\mathop{\rm{Ad}}}
\newcommand{\ad}{\mathop{\rm{ad}}}
\newcommand{\rI}{\mathrm{I}}

\newcommand{\Ind}{{\rm{Ind}}}

\newcommand{\Sym}{\mathop{\rm{Sym}}}

\newcommand{\rank}{\mathop{\rm{rank}}}

\newcommand{\rS}{\mathrm{S}}
\newcommand{\rU}{\mathrm{U}}

\newcommand{\1}{\mathbf{1}}
\newcommand{\Gr}{\mathop{\rm{Gr}}}

\newcommand{\inner}[2]{\langle#1,#2\rangle} 
\newcommand{\innerl}[2]{\langle#1,#2\rangle_\l}

\newcommand{\herm}[2]{\langle#1,#2 \rangle}

\newcommand{\C}{\ensuremath{\mathbb C}}
\renewcommand{\H}{\ensuremath{\mathbb H}}
\newcommand{\K}{\ensuremath{\mathbb K}}

\newcommand{\R}{\ensuremath{\mathbb R}}

\newcommand{\N}{\ensuremath{\mathbb N}}

\renewcommand{\l}{\lambda}
\renewcommand{\a}{\alpha}
\renewcommand{\b}{\beta}

\newcommand{\fa}{\mathfrak{a}}
\newcommand{\amin}{\mathfrak{a}_{\text{min}}}
\newcommand{\fb}{\mathfrak{b}}

\newcommand{\fp}{\mathfrak{p}}
\newcommand{\fe}{\mathfrak{e}}
\newcommand{\ff}{\mathfrak{f}}
\newcommand{\frakg}{\mathfrak{g}}
\newcommand{\fg}{\mathfrak{g}}
\newcommand{\fh}{\mathfrak{h}}
\newcommand{\fk}{\mathfrak{k}}
\newcommand{\fl}{\mathfrak{l}}
\newcommand{\fq}{\mathfrak{q}}
\newcommand{\fm}{\mathfrak{m}}
\newcommand{\fn}{\mathfrak{n}}
\newcommand{\ft}{\mathfrak{t}}
\newcommand{\nmin}{\mathfrak{n}_{\text{min}}}
\newcommand{\fs}{\mathfrak{s}}
\newcommand{\fu}{\mathfrak{u}}
\newcommand{\fz}{\mathfrak{z}}

\newcommand{\bx}{\mathbf{x}}

\newcommand\SO{\mathrm{SO}}
\newcommand\SL{\mathop{\rm{SL}}}
\newcommand\GL{\mathop{\rm{GL}}}
\DeclareMathOperator{\SU}{SU}

\newcommand\upS{{\rm{S}}}
\newcommand\upO{{\rm{O}}}

\newcommand{\wG}{\widetilde{G}}

\newcommand{\wt}{\widetilde}
\newcommand{\wtt}{\widetilde{\tau}}
\newcommand{\wKL}{\widehat{K}_L}

\renewcommand{\phi}{\varphi}

\newcommand{\ip}[2]{\langle #1,#2 \rangle}

\newcommand{\wE}{\widehat{\mathcal{E}}}

\newcommand{\cB}{\mathcal{B}}
\newcommand{\cC}{\mathcal{C}}
\newcommand{\cD}{\mathcal{D}}
\newcommand{\cE}{\mathcal{E}}
\newcommand{\cH}{\mathcal{H}}

\newcommand{\cN}{\mathcal{N}}
\newcommand{\cO}{\mathcal{O}}
\newcommand{\cS}{\mathcal{S}}
\newcommand{\cW}{\mathcal{W}}

\def\sideremark#1{\ifvmode\leavevmode\fi\vadjust{\vbox to0pt{\vss \hbox to 0pt{\hskip\hsize\hskip1em\vbox{\hsize2cm\tiny\raggedright\pretolerance10000  \noindent #1\hfill}\hss}\vbox to8pt{\vfil}\vss}}}
        
\newcommand{\WH}{W_{H\cap K}}
\newcommand{\W}{W_K}
               
\def\LB{\Lambda^+(\cB )}

\newcommand{\so}{\mathfrak{so}}
\renewcommand{\sp}{\mathfrak{sp}}
\newcommand{\Sp}{\mathrm{Sp}} 
\renewcommand{\sl}{\mathfrak{sl}}
\newcommand{\gl}{\mathfrak{gl}}
\newcommand{\su}{\mathfrak{su}}
\newcommand{\pmax}{\mathfrak{p}_{\rm{max}}}
\newcommand{\Pmax}{P_{\rm{max}}}
\newcommand{\Pmin}{P_{\rm{min}}}
\newcommand{\bN}{\overline{N}}
\newcommand{\bP}{\overline{P}_{\rm{max}}}
\newcommand{\bn}{\overline{n}}
\newcommand{\pr}{\mathrm{pr}}
\newcommand{\Cos}{\mathrm{Cos}}

\begin{document}

\mathclass{Primary 22E46; Secondary 43A85, 57S25.}

\abbrevauthors{J. M\"{o}llers, G. \'{O}lafsson and B. {\O}rsted}
\abbrevtitle{The Berezin form and reflection positivity}

\title{The Berezin form on symmetric $R$-spaces and reflection positivity}

\author{Jan M{\"o}llers}
\address{Department Mathematik, FAU Erlangen--N\"{u}rnberg\\
Cauerstr. 11, 91058 Erlangen, Germany\\
E-Mail: moellers@math.fau.de}

\author{Gestur {\'O}lafsson}
\address{Department of Mathematics, Louisiana State University\\
Baton Rouge, LA 70803, USA\\
E-Mail: olafsson@math.lsu.edu}

\author{Bent {\O}rsted}
\address{Institut for Matematiske Fag, Aarhus Universitet\\
Ny Munkegade 118, 8000 Aarhus C, Denmark\\
E-Mail: orsted@math.au.dk}

\maketitlebcp

\begin{abstract}
For a symmetric $R$-space $K/L=G/P$ the standard intertwining operators
provide a canonical $G$-invariant pairing between sections of line
bundles over $G/P$ and its opposite $G/\overline{P}$. Twisting this
pairing with an involution of $G$ which defines a non-compactly causal
symmetric space $G/H$ we obtain an $H$-invariant form on sections of
line bundles over $G/P$. Restricting to the open $H$-orbits in $G/P$
constructs the Berezin forms studied previously by G. van Dijk, S. C. Hille and V. F. Molchanov. We determine for which $H$-orbits in $G/P$ and for which line bundles the Berezin form is positive semidefinite, and in this case identify the corresponding representations of the dual group $G^c$ as unitary highest weight representations. We further relate this procedure of passing from representations of $G$ to representations of $G^c$ to reflection positivity.
\end{abstract}

\tableofcontents

\section*{Introduction}

\addcontentsline{toc}{section}{Introduction}

The notion of \textit{reflection positivity} appeared first as one of the Osterwalder--Schrader axioms in constructive quantum
field theory, see \cite{OS73,OS75}. In this connection it can be viewed as a tool to transform a quantum mechanical system to a quantum field theoretical
system via the Osterwalder--Schrader quantization process. From the point of view of representation theory, it is a receipt to construct from a representation of
the Euclidean motion group a unitary representation of the Lorentz group. In this form, reflection positivity
can be formulated more generally as transforming a representation of one Lie group $G$ to a unitary
representation of another Lie group $G^c$. Here the groups $G$ and $G^c$ are
connected via {\it Cartan duality}. For that let $G$ be a, say connected, Lie group
with an involution $\tau : G\to G$. The involution gives rise to an involution
$\tau : \fg \to \fg$ on the Lie algebra by differentiation. The Lie algebra $\fg$ then decomposes as $\fg=\fh\oplus \fq$ where
$$ \fh =\{X\in\fg\mid \tau (X)=X\} \qquad \mbox{and} \qquad \fq=\{X\in\fg\mid \tau (X)=-X\}. $$
The commutation relations $[\fh,\fh],[\fq,\fq]\subseteq \fh$ and $[\fh,\fq]\subseteq \fq$ imply that $\fg^c:=\fh\oplus i\fq$
is also a Lie algebra. Note that both $\fg$ and $\fg^c$ are real forms of the same complex Lie algebra $\fg_\C$.
One then defines $G^c$ to be a connected Lie group with Lie algebra $\fg^c$.

The first articles to address this idea were \cite{LM75,J86,J87,S86}. The first three of these papers deal
with the problem of integrating an infinitesimally unitary representation of $\fg^c$ to a unitary representation
of $G^c$. Subsequently, R. Schrader~\cite{S86} used this idea for the first time in the context of simple Lie groups. More precisely, he applies
reflection positivity and the integration results from \cite{LM75} to a degenerate spherical principal series representation of
$G=\SL(2n,\C)$. This constructs a unitary representation of the dual group $G^c=\SU(n,n)\times\SU(n,n)$, but he does not identify this representation. That question was taken up in \cite{JO98,JO00,O00} where Schrader's idea was further generalized to all simple groups such
that $G^c$ is semisimple and of Hermitian type. Special attention was paid to simple groups such that the corresponding
bounded symmetric domain is of tube type $\R^p+i\Omega$ and $H=G^\tau$ is locally isomorphic to the automorphism group of
the open symmetric cone $\Omega\subseteq\R^p$. It was shown that if one starts with a degenerate principal
series representation of $G$, then the process of reflection positivity results in an irreducible highest weight
representation of $G^c$. The authors were not aware of the fact that much earlier T. Enright had discussed in
\cite{E83} a method to transform a degenerate principal series representation of $G=G'_\C$ to a
highest weight representation of $G^c=G'\times G'$, a special case of the above setting. But Enright's methods
are algebraic in nature and not related to the idea of reflection positivity.

In \cite{O00}, and to some extent also in \cite{JO98,JO00}, it was pointed out that the ideas of applying
reflection positivity to the representation theory of semisimple groups are closely related to several other ideas that were floating around at the
same time, in particular the connection to the Segal--Bargman transform \cite{OO96} and the
Berezin transform and canonical representations developed by G. van Dijk, S. C. Hille and others,
see \cite{B75,vDH97a, vDH97b,vDM98,vDM99,vDP99,FP05,H99}. This connection is one of the
main topics in this article. Here we review previous results and complete
the picture by giving a full answer to the positivity question for the Berezin form.\\

We start by recalling some basic facts about the types of symmetric spaces that are of importance for this article (see Section~\ref{Se:Rspaces}). More precisely, we discuss non-compactly causal symmetric spaces
(see \cite{HO97}), symmetric $R$-spaces (see \cite{HS97,N65,T79,T87}), and bounded symmetric domains (see \cite{KW65a,KW65b,W72}). This discussion includes the maximal parabolic subgroups $P_{\text{max}}=MAN$ of $G$ that will play an important role in the rest of the article. In our situation $N$ is abelian, $MA=G^{\theta\tau}$ is the centralizer of an element $X_0\in \fg$ with the property that the Lie algebra of $N$ is the $+1$-eigenspace of $\ad X_0$, and $A=\exp (\R X_0)$. Here $\theta$ is a Cartan involution commuting with $\tau$. The (generalized) flag manifold $\cB=G/\Pmax$ is called a symmetric $R$-space.

In Section \ref{Se:Principal} we recall the construction of the spherical degenerate principal series representations $\pi_\lambda$ of $G$, induced from the maximal parabolic $\Pmax$, and the associated standard intertwining operators $J(\lambda)$. There are various ways to realize the representations $\pi_\lambda$. For our purpose the two canonical ways are to either realize $\pi_\lambda$ as acting on $L^2(\cB)$ or on a weighted $L^2$-space on $\bN$. For practical purposes it is more convenient to consider smooth induction which realizes $\pi_\lambda$ as a representation on $C^\infty (\cB)$ or a subspace of $C^\infty (\bN )$. In particular, this is necessary when considering the meromorphic extension of the intertwining operators as a function $\lambda \mapsto J(\lambda)$. We finish this section by recalling from the literature the interval where the representations $\pi_\lambda$ give rise to irreducible unitary representations, the \textit{degenerate complementary series representations}. The material 
in this section is mostly standard, and for the special case where $\cB$ is a Grassmannian the corresponding results can be found in \cite{OP12}. In fact, this example serves as an illustration throughout the whole paper. Note that in \cite{OP12} some additional results are obtained that we do not mention here. This includes the
calculation of the eigenvalues of $J(\lambda )$ on each of the $K$-types using the \textit{spectrum generating operator} from \cite{BOO96}, which was generalized to all symmetric $R$-spaces in \cite{MS14}. The statement, however, splits into various cases, so we refer the reader to \cite{OP12,MS14} for details.

We introduce the \textit{Berezin kernel}, the \textit{Berezin transform} and the associated \textit{Berezin form} $\ip{\cdot }{\cdot }_\lambda$ in Section \ref{Se:Berezin}, following the idea of Hille \cite{H99}. 
In Proposition \ref{prop 4.3} we show that
\[\ip{\pi_\lambda (g)f}{h}_\lambda = \ip{f}{\pi_{\bar \lambda}(\tau (g^{-1}))h}_\lambda\, .\]
In particular, if $\lambda$ is real then the Berezin
form is $H$-invariant and hence, assuming its positivity, defines a
unitary representation of $H$. This is
the \textit{canonical representation}. A second important result in this section is Lemma \ref{le:4.4} where we express the Berezin form in the $\bN$-realization. More precisely, for compactly supported functions $f$ and $h$ on $\bN$ we show that
\[\ip{f}{h}_\lambda =\int_{\bN\times \bN} \kappa_\lambda (x,y)f_\lambda (x)\overline{h_\lambda (y)}\, dx\,dy\]
where the kernel $\kappa_\lambda$ is explicitly given by the $A$ projection in the triangular decomposition $\bN MA N\subseteq G$, and $f_\lambda$ resp. $h_\lambda$ is given by multiplying $f$ resp. $h$, by
a certain positive function depending on $\lambda$. This is a fundamental expression as we move on to reflection positivity where one needs to determine where $\kappa_\lambda$, or rather its twisted version $\kappa_\lambda\circ (\tau\times\id)$, is positive definite. The results is also needed for identifying the resulting representation $\pi_\lambda^c$ of the dual group $G^c$.

Section \ref{se:openOrb} is devoted to a description of the open $H$-orbits in $\cB$. This is done in 
Theorem \ref{th:double} where we express the open orbits using a maximal set of 
strongly orthogonal roots related to a minimal parabolic subgroup. In particular, the number of open orbits is equal to $r+1$, where $r=\rank(H/K\cap H)$. We then show in Proposition \ref{pro:5.5} that the
open orbits are symmetric spaces $\cO_j=H/H^{\sigma_j}$ where $\sigma_j$ is an explicitly given involution ($0\leq j\leq r$).
The orbit $\cO_0$ through the base point $b_0:=eP_{\text{max}}$ is $H/(H\cap K)$ which is a Riemannian symmetric space. Most of the other orbits are non-Riemannian.

We give the basic definitions related to \textit{reflection positivity} in Section \ref{sec:ReflectionPositivity}. For this consider the smooth representation $\pi_\lambda$ of $G$ on $\cE =C^\infty (\cB)$. For each of the
open orbits let $\cE_{j,+}=C^\infty_c(\cO_j)$. Then the Berezin form $\ip{\cdot}{\cdot}_\lambda$ is
non-negative on $\cE_{j,+}$ if and only if the restriction of the Berezin
kernel $\kappa_\lambda$ is positive definite on $\cO_j\times \cO_j$. 
In this case we let $\cN_j$ be the radical of $\ip{\cdot }{\cdot}_\lambda$ restricted to $\cE_{j,+}\times \cE_{j,+}$ and
let $\widehat{\cE}_j$ be to completion of $\cE_{j,+}/\cN_j$. Then $\wE_j$ is a Hilbert space which carries a unitary representation of $H$ and an infinitesimally unitary representation of the dual Lie algebra $\fg^c$ which one wants to integrate to $G^c$ (or a covering of $G^c$).

This is then applied to our situation in following two sections. In Section \ref{sec:HighestWeightRepresentations} we
discuss the Riemannian symmetric orbits $\cO_0=H\cdot b_0$ and, in case that $G^c$ is of tube type, also the
conjugate orbit $\cO_r$. We start by recalling the notion of \textit{(unitary) highest weight representations} $(\rho_\mu,\cH_\mu^c)$.
We show that the kernel $\kappa_\lambda$ is the restriction of the reproducing kernel $K_\mu$ of the unitary representation $\rho_\mu$, where $\mu$ and $\lambda$ are related by $\lambda = -\mu+\rho$. 
Furthermore, it is well known that the orbit $\cO_0$ is a totally real submanifold of the complex manifold $G^c/K^c$ with $K^c\subseteq G^c$ being a maximal compact subgroup. In fact, $\tau$ defines a complex conjugation on $G^c/K^c$ with fixed point set $\cO_0$.
It follows that $\kappa_\lambda$ restricted to $\cO_0\times \cO_0$ is positive definite if and only if the highest weight
representation $\rho_\mu$ of $G^c$ is unitary. This
result, which is stated as Proposition \ref{prop7.3}, gives a complete answer to the positivity question for the Riemannian orbits as all other orbits are non-Riemannian, see also \cite{JO98,JO00}. In Section \ref{se:7.3} we further show that
\[T_\mu f (z):=\int_{\cO_0} K_\mu (z,x)f(x)\, dx\, ,\quad f\in \cO_c(\cO_0)\]
defines an isometry $\cE_{0,+}/\cN \to \cH_\mu$ which then extends to an
unitary isomorphism $\widehat{\cE}_0\to\cH_\mu^c$ intertwining $\pi_\lambda^c$ and $\rho_\mu$. Similarly, we construct in Section \ref{se:7.4} a unitary intertwining operator from $\widehat{\cE}_0$ into the holomorphic discrete series of the symmetric space $G^c/\widetilde{H}$.

Reflection positivity related to the Riemannian open orbit $\cO_0$ has been observed earlier, but the non-Riemannian orbits have not been treated so far. This we do in Theorem \ref{thm:8.6} where we show that for all of those orbits
the Berezin kernel is \textit{not} positive, unless it is trivial (and the process of reflection positivity constructs the trivial representation of $G^c$). This is accomplished by a rank two reduction using the pairs $(\sl(3,\R),\so(1,2))$ and $(\sp(2,\R),\gl(2,\R))$, see Lemma \ref{lem:SubalgebrasSL3Sp2} which is in fact interesting and useful in itself.

Finally, in Section~\ref{sec:HLS} we discuss a recent application of reflection positivity for the special case $G=\SO(n+1,1)$, namely a new proof of the sharp Hardy--Littlewood--Sobolev inequality by R. Frank and E. Lieb~\cite{FL10}. Since the proof uses special cases of a few statements that hold in the more general context of symmetric $R$-spaces, one may wonder whether it can be modified to establish a theory of sharp Hardy--Littlewood--Sobolev inequalities in this more general setting.

\subsection*{Acknowledgements} The authors would like to thank A. Pasquale for helpful discussions during the initial stage of this project.

\section{Symmetric $R$-spaces, non-compactly causal symmetric spaces, and bounded symmetric domains}\label{Se:Rspaces}

In this section we recall some basic facts about symmetric spaces, in particular the notion of non-compactly causal symmetric
spaces and symmetric $R$-spaces. Our standard references are \cite{H78,HO97} for non-compactly causal symmetric spaces, \cite{HS97,K00,Lo77,N65,T79,T87} for symmetric $R$-spaces, and \cite{KW65a,KW65b,W72} for bounded symmetric domains.
 
\subsection{Non-compactly causal symmetric spaces}\label{sec:NCCSymmetricSpaces}

Let $G$ be a connected non-compact semisimple Lie group with Lie algebra $\mathfrak g$. We assume that
$G$ is contained in a connected complex Lie group $G_\C$ with Lie algebra $\fg_\C=\fg\otimes_\R \C$. Then
the center of $G$ is finite.
For any closed subgroup $S\subseteq G$ with Lie algebra $\fs$ we denote by $S_\C$ the complex subgroup of
$G_\C $ generated by $S$ and $\exp(\fs_\C)$. Then the Lie algebra of $S_\C$ is  $\fs_\C$.

Let $\theta$ be a Cartan involution on $G$
and $\frakg=\fk \oplus \fp$ the corresponding Cartan decomposition of $\frakg$. Here, and in the following, 
if $\sigma$ is an automorphism of $G$, we denote by the same symbol $\sigma$ the derived automorphism of $\frakg$. 
Set $G^\sigma=\{a\in G\mid \sigma(a)=a\}$ and 
$\frakg^\sigma=\{X\in \frakg\mid \sigma (X)=X\}$.
Let $K=G^\theta$.  Then $K$ is connected, has Lie algebra $\fk$, and is a maximal compact subgroup of $G$. 
Let $\tau$ be a nontrivial involution of $G$ which commutes with $\theta$. We say that $(G,\tau )$ and $(\fg,\tau)$ are
\textit{symmetric pairs}. Let $\frakg=\fh\oplus \fq$ be the eigenspace decomposition of $\frakg$ with respect
to the derived involution $\tau$, then 
\[\frak g=\fk \cap \fh \oplus \fk\cap \fq \oplus \fp \cap \fh \oplus \fp \cap  \fq\, .\] 
If $H$ is an open subgroup of $G^\tau$, then $G/H$ is said to be an \textit{(affine) symmetric space}. Let $L=K\cap H$ and $\fl=\fk\cap\fh$, then $L$ is a maximal compact subgroup of $H$. Symmetric
pairs and spaces always come in pairs $(\fg,\tau )$ and $(\fg, \tau\theta )$, and we abbreviate $\wtt=\tau\theta$.
We put
$$ \fg_0:=\fg^{\wtt}=\fl \oplus \fp\cap \fq \qquad \mbox{and} \qquad G_0:=L\exp(\fp\cap \fq), $$
then $G/G_0$ is also an affine symmetric space.

The symmetric space $G/H$ is said to be irreducible if $\{0\}$ and $\frakg$ are the only $\tau$-invariant ideals of $\fg$. We will
always assume that $G/H$ is irreducible. In that case either $G$ is simple or of the
form $G=G'\times G'$ with $\tau (a, b) = (b,a)$, $\wtt (a, b )=(\theta_1(b), \theta_1(a))$  and $H=\{(a , a)
\mid a\in G'\}\simeq G'\simeq G_0$. In this case $G/H\simeq G'$ via the map 
$(a, b)H\mapsto ab^{-1}$ and the action of $G$ on $G/H$ is transformed into the left-right action of $G'\times G'$ on $G'$: $(a , b)\cdot x=axb^{-1}$.
 
We recall that an element $X \in \frakg$ is called \textit{hyperbolic} if the operator 
$\ad(X): Y\mapsto [X,Y]$ on $\frakg$ is semisimple with real eigenvalues. A subset of $\frakg$ is said to be \textit{hyperbolic} if it consists
of hyperbolic elements. An irreducible symmetric space $G/H$ is said to be \textit{non-compactly causal} if there exists a
non-empty open hyperbolic $H$-invariant convex cone $C \subset \fq$ containing no affine line. This is equivalent to the existence of a
non-zero hyperbolic element
$$ X_0\in (\fp\cap \fq)^L=\{Y\in\fp\cap \fq\mid (\forall k\in L)\,\, \Ad (k)Y=Y\}. $$

\begin{Rem}
We note that being a non-compactly causal symmetric space does not only depend on the infinitesimal data $(\fg,\fh)$, it might also depend on $H/H_0$, where $H_0$ denotes the connected component containing the identity. Assume that $G$ is simple, $\tau$ an involution on $G$  commuting with  $\theta$ and assume that the Cartan
involution is an inner automorphism. Let $G_1=\Ad (G)\subset \GL (\fg)$ and let $H_1$ be the connected subgroup with Lie algebra $\fh$. The
involution $\tau$ defines an involution on $G_1$ that we denote by $\tau_1$. It is given by
$\tau_1(a)=\tau\circ a \circ \tau$. It is clear that $\theta\in G_1^{\tau_1}$. Hence, even if $G_1/H_1$ is non-compactly causal, the space $G_1/G_1^{\tau_1}$ can
never be non-compactly causal. A typical example is $(\SO (1,2),\SO(1,1))$ and more generally  Cayley type symmetric spaces, see bellow for the definition. We note that
in this case $\tau $ and $\wtt$ are conjugate. So in particular $H$ and $G_0$ are conjugate and hence isomorphic.
\end{Rem}

In the following we assume that $G/(G^\tau)_0$ is an irreducible non-compactly causal symmetric space. We now choose $(G^\tau)_0\subseteq H\subseteq G^\tau$ maximal such that $G/H$ still is non-compactly causal. Fix a non-zero hyperbolic element $X_0\in(\fp\cap \fq)^{\fk\cap\fh}$ and let
$L=Z_K(X_0)$ and $H=L (G^\tau)_0$. At this point it is not yet clear that $H$ is a group, but this will follow later.

We now recall some structure theory for non-compactly causal symmetric spaces from \cite{HO97}, to which we refer the interested reader for details. We can, and always will, normalize $X_0$ so that $\ad(X_0)$ has eigenvalues $0,1$ and $-1$.
Let $\frakg_0$, $\frakg_1$ and $\frakg_{-1}$ denote the corresponding eigenspaces of $\ad(X_0)$ in $\frakg$.
Then
$$ \frakg =\frakg_{-1} \oplus \frakg_0 \oplus \frakg_1 $$
defines a $3$-grading of $\frakg$.
Note that $\frakg_0=\fk\cap\fh\oplus\fp\cap\fq=\frakg^{\wtt}$, and hence the definition of $\fg_0$ agrees with the previous one. Furthermore,
$\ad X_0 : \fk\cap \fq\to \fp\cap \fh$ is a linear isomorphism with inverse $\ad X_0|_{\fp\cap \fh}$. It
follows in particular that $\fk_\C \simeq \fh_\C$ as $L$-modules. We will see in a moment, that those Lie algebras are in fact conjugate.

It follows from the definition that $\wtt (X_0)=X_0$. Thus $\wtt $ defines by restriction an involution on $\fg_{1}$ and $\fg_{-1}$ with possible eigenvalues $\pm1$. Assume that $X\in\fg_{\pm1}$ with $\wtt (X)=X$, then $X+\theta (X)=X+\tau (X)\in \fk\cap \fh\subseteq\fg_0$. In particular $0=[X_0,X+\theta (X)]= X-\theta (X)$. Thus $X=0$ and we have shown that
$\theta|_{\fg_{\pm1}}= -\tau|_{\fg_{\pm1}}$. It also follows that
$$ \fk\cap \fq=\{X+\theta (X)\mid
X\in \fg_1\} \qquad \mbox{and} \qquad \fh\cap \fp = \{X-\theta (X)\mid X\in \fg_1\}. $$

Set 
\begin{equation}\label{eq:psi-wttau}
\psi:=\Ad\left(\exp \frac{i\pi}{2}X_0\right)=\exp\left(\frac{i\pi}{2} \ad X_0\right),
\end{equation}
then
\[\psi|_{\fg_0}=\id\, \quad \psi |_{\fg_{\pm 1}}=\pm i\,\id_{\fg_{\pm1}}\, .\]
It follows that $\psi^2 =\wtt$, in particular we have $L\subset G^{\wtt}$.
As $L\subset K$ and $\tau =\theta\wtt$ we obtain $H \subseteq G^\tau$. This shows
that $L$ normalizes $\fh$ and hence $L$ normalizes $(G^\tau)_0=H_0$. It now follows that $H$ is in fact a group and $G/H$ a non-compactly causal symmetric space. We also note that $G^\tau/H=K^\tau/L$. As
$\dim(\fp\cap\fq)^L=1$ and $L$ normalizes $K^\tau$ it follows that for $k\in K^\tau$ we have $\Ad (k ) X_0=\pm X_0$. Hence, 
either $H=G^\tau$ or $G^\tau/H$ is a two element group. If $G$ is the adjoint group then the latter case occurs if and only if the Cartan involution is inner. 

Finally, we note that  $\psi : K_\C \to H_\C$  is an analytic isomorphism inducing an
isomorphism $K_\C /L_\C\to H_\C /L_\C$.

Let $\pmax=\fg_0\oplus \fg_{1}$, then $\pmax$ is a maximal parabolic subalgebra with corresponding maximal 
parabolic subgroup $\Pmax=N_G(\pmax )=G_0\exp(\fg_1)$. We have that
\[\cB := G/\Pmax \simeq K/L\]
is a compact symmetric space. Let $b_0=e\Pmax\in \cB$ denote the base point. 

\begin{Lemma}[{see \cite[Lemma 5.1.1]{HO97}}]\label{lem:InvolutionsOnPmax}
We have
\begin{align}
\theta(\Pmax)&=\tau(\Pmax)=G_0 \exp(\frakg_{-1})\\
\wt\tau(\Pmax)&=\Pmax\,.
\end{align}
Moreover, in the Langlands decomposition $\Pmax=MAN$ of $\Pmax$ we have $A=\exp(\R X_0)$, $MA=G_0$ and $N=\exp\frakg_1$. 
Furthermore, $\Pmax \cap H=K \cap H=L$.
\end{Lemma}

We will use the notation $\fn=\fg_{1}$, $\overline\fn=\fg_{-1}$, $N=\exp(\fn)$, $\overline N=\theta N =\exp \overline\fn$ and $\bP=\theta\Pmax=MA\bN$. Note that $\bN \Pmax$ is open and dense in $G$ and that
\[\bN \to \cB, \quad \overline n\mapsto \overline n\cdot x_0\]
is a diffeomorphism onto an open dense set. More precisely, the map
$$ \bN\times M\times A\times N\to G, \quad (\bn,m,a,n)\mapsto\bn man $$
is a diffeomorphism onto an open dense subset of $G$, and we write for $g\in \bN \Pmax$:
\[g=\overline \nu (g)\mu (g)\alpha (g)\nu (g) \in \bN MAN. \]
Then the almost everywhere defined action of $G$ on $\bN\simeq\bN\cdot b_0\subseteq\cB$ is given by
$$ g\cdot\overline{n} = \overline{\nu}(g\overline{n}), \qquad g\in G,\overline{n}\in\overline{N}. $$

For $g\in G$ we further write
$$g=k(g)m(g)a(g)n(g)\in KMAN=G. $$
Then $a(g)$ and $n(g)$ are well-defined analytic functions of $g$, but $k(g)$ and $m(g)$ are only defined modulo $L$. However, the map
$$ G\times\cB\to\cB, \quad (g,b)\mapsto g\cdot b:=k(g)b $$
is well-defined and equal to the left-action of $G$ on $\cB=K/L\simeq G/\Pmax$.

\subsection{Symmetric $R$-spaces}

Symmetric $R$-spaces are compact symmetric spaces admitting a non-compact group of transformations. In short, we will call an irreducible compact connected symmetric space $K/L$ a \textit{symmetric $R$-space} if there exists a non-compact simple Lie group $G$ acting transitively on
$K/L$ such that $K/L=G/\Pmax$ with $\Pmax=MAN$ a maximal parabolic subgroup with abelian nilradical $N$. As $\Pmax$ is maximal
it follows that $\fa$, the Lie algebra of $A$, is one dimensional. Further, since the Lie algebra $\fn$ of $N$ is abelian, there exists $X_0\in \fa$ such that $\fn$ is the eigenspace of $\ad X_0$ with eigenvalue $+1$. As $\overline \fn=\theta \fn$ it follows that
 $\fg=\fg_{-1} \oplus \fg_0 \oplus \fg_{+1}$ is $3$-graded with $\fg_{-1}= \overline\fn$, $\fg_0 = \fm\oplus \fa$, and $\fg_{1}
 =\fn$. 

On the other hand, if $\fg = \fg_{-1}\oplus \fg_0\oplus \fg_1$ is a simple 3-graded Lie algebra, then
$[\fg_0,\fg_0], [\fg_{-1},\fg_1]\subseteq \fg_0$ and $[\fg_0,\fg_{\pm 1}]\subseteq \fg_{\pm 1} $. Furthermore
$[\fg_1,\fg_1]\subset \fg_2=\{0\}$ and similarly it follows that $\fg_{-1}$ is abelian. This in particular implies that
\[[\fg_{-1}\oplus \fg_1,\fg_{-1}\oplus \fg_1]=[\fg_1,\fg_{-1}]\subseteq \fg_0\, .\]
It follows that $(\fg,\wtt)$ is a symmetric pair where the involution $\wtt $ is given by
\[\wtt|_{\fg_0}=\id\quad\text{ and }\quad \wtt|_{\fg_{1}\oplus \fg_{-1}}=-\id\, .\]
As $\fg_{\pm 1}$ are $\fg_0$-invariant it follows from \cite[Lem. 1.3.4]{HO97} that there exists a hyperbolic element $X_0\in \fg_0$, which one can
assume to be in $\fg_0\cap \fp$, such that
$\ad X_0$ has eigenvalues $0,1,-1$ and
\[\fg_0=\fg(\ad X_0,0),\quad \text{and}\quad \fg_{\pm 1} =\fg(\ad X_0,\pm 1)\, .\]
Then, as we observed in Section~\ref{sec:NCCSymmetricSpaces}, $\wtt=\exp(i \pi \ad (X_0))$. Furthermore, with $\tau =\wtt\theta$, the symmetric space $G/(G^\tau)_0$ is non-compactly causal. Thus, the irreducible 
non-compactly causal symmetric spaces $G/H$ with $H$ connected are in one-to-one correspondence to the irreducible symmetric $R$-spaces $G/\Pmax$, or equivalently the $3$-graded simple Lie algebras $\fg$.

We note that the symmetric spaces $G/G_0$ are the simple  \textit{parahermitian symmetric spaces}, see \cite{K85,KA88}.

\subsection{Bounded domains}\label{sec:BoundedDomain}

Given a semisimple symmetric pair $(\fg,\tau)$ as in Section~\ref{sec:NCCSymmetricSpaces} one can construct a new semisimple
symmetric pair $(\fg^c,\tau)$, by defining $\fg^c:=\fh\oplus i\fq\subseteq\fg_\C$ and denoting by $\tau$ also the complex linear extension of $\tau$ to $\fg_\C$ as well as its restriction to $\fg^c$. This process is called \textit{c-duality}. Note that $(\fg^c)^\tau=\fh$ and $\wtt|_{\fg^c}$ is a Cartan involution of $\fg^c$. The corresponding maximal compact subalgebra $\fk^c$ of $\fg^c$ is given by $\fk^c=\fk\cap \fh\oplus i(\fp\cap \fq)$, and the Cartan decomposition of $\fg^c$ is $\fg^c=\fk^c\oplus\fp^c$ with $\fp^c=\fp\cap \fh \oplus i(\fk\cap \fq)$. In particular,
$c$-duality interchanges the elliptic and hyperbolic directions and $\fk^c_\C=\fg_{0,\C}$. The element $Z_0=iX_0$ is a central element in the maximal compact subalgebra $\fk^c$, and the eigenvalues of $\ad Z_0$ are $0$ with eigenspace $\fk^c$ and $\pm i$ with (complexified) eigenspaces $\fp^c_{\pm}=\fg_{\pm1,\C}$. 

We denote by $G^c$ the analytic subgroup in $G_\C$ with Lie algebra $\fg^c$ and by $\wG^c$ its universal covering group. We see that $G^c/K^c$ is a bounded symmetric domain which can be realized as an open $G^c$ orbit in $G_\C/K^c_\C P^c_+=G_\C/P_{\rm{max},\C}$, where $P^c_\pm=\exp(\fp^c_\pm)\subseteq G_\C$. 
The involutions $\tau$ and $\tilde \tau$ extends to holomorphic involutions on $G_\C$ and then by restriction to involutions on $G^c$. Both
involutions leave $K^c$ invariant and hence define involutions on $G^c/K^c$.  We use the same notation for these involutions, and it will be clear from the context on which spaces these involutions act.
As the complex structure on $G^c/K^c$ is given by $\ad Z_0$ and $\tau Z_0=-Z_0$ it follows that $\tau : G^c/K^c\to G^c/K^c$ is an antiholomorphic involution. In particular $H\cdot x_0=(G^c/K^c)^\tau $ is a totally real submanifold, where $x_0=eK^c$ is the base point. Note that $\tau (P^c_\pm)
=P^c_\mp$, hence $\tau$ does not define an involution on the flag manifold $G_\C/P_{\rm{max},\C}$. Let $\sigma^c:
\fg_\C\to \fg_\C$ be the conjugation with respect to $\fg^c$ and let $\eta =\tau \sigma^c=
\sigma^c\tau$. Then $\eta$ defines a conjugation on $G_\C/P_{\rm{max},\C}$ which extends the involution
$\tau $ on $G^c/K^c$.

The map $P^c_-\times K_\C^c\times P^c_+\to G_\C,(p_-,k^c,p_+)\mapsto p_-k^cp_+$ is a diffeomorphism onto an open dense subset of $G_\C$ and we write $g=p_-(g)k^c(g)p_+(g)\in P^c_-K_\C^c P^c_+$ for the corresponding triangular decomposition. For future reference we note the following fact:

\begin{Lemma}\label{lem:OpenDenseDecompositionsAgree}
Let $g\in G$. Then, whenever defined, $\overline \nu (g)=p_-(g)$, $\mu (g)\alpha (g)=k^c(g)$, and $\nu (g)=p_+(g)$.
\end{Lemma}

\begin{proof}
This follows from $P^c_-\cap G=\overline N$, $K^c\cap G=G_0$ and $P^c_+\cap G=N$.
\end{proof}

\subsection{The classification}

We end this section with a classification of all irreducible non-compactly causal symmetric spaces $G/H$ in terms of the Lie algebras $\fg$, $\fg^c$ and $\fh$. Note that $\fg$ is always a simple real Lie algebra, but $\fg_\C$ is not necessarily a simple complex Lie algebra. This is precisely the case if $\fg$ does not have a complex structure, and we therefore divide the classification into two tables, depending on whether $\fg$ has a complex structure or not (see Table~\ref{tab:ClassificationCplx} and \ref{tab:ClassificationReal}). Note that if $\fg$ does have a complex structure, then $(\fg^c,\fh)\simeq(\fh\oplus\fh,\fh)$, the so-called \textit{group case}.

The cases where $\fh$ has a center are called \textit{Cayley type}. Those are exactly the cases where $\fg_0\simeq \fh$. This is further equivalent to $G^c/K^c$ being a tube domain $T_\Omega =\R^n +i \Omega$ with $\Omega\subseteq\R^n$ a symmetric cone and $H_0=\mathrm{Aut}(\Omega)_0$, the automorphism group of the cone.

For each symmetric space $G/H$ we also list the rank of the non-compact Riemannian symmetric space $H/L$ which equals the rank of the compact Riemannian symmetric space $K/L$. In the tables we always assume that $n\geq1$ and $p,q\ge1$.

\begin{table}
\centering
\begin{tabular}{|c|c|c|c|c|}
\hline
& $\fg$ & $\fg^c$&$\fh$ & $\rank H/L$\\
\hline
A & $\sl (p+q,\C)$ & $\su (p,q)\times \su (p,q)$& $\su (p,q)$ & $\min(p,q)$\\
BD & $\so(n+2,\C)$& $\so (2,n)\times \so (2,n)$&$\so (2,n)$ & $2$\\
C & $\sp (n,\C)$& $\sp (n,\R)\times \sp (n,\R)$ & $\sp (n,\R)$ & $n$\\
D & $\so (2n,\C)$ & $\so^*(2n)\times \so^* (2n)$& $\so^*(2n)$ & $\lfloor n/2\rfloor$\\
E$_6$ & $\fe_6(\C)$& $\fe_{6(-14)}\times \fe_{6(-14)}$& $\fe_{6(-14)}$ & $2$\\
E$_7$ & $\fe_7(\C)$& $\fe_{7(-25)}\times \fe_{7(-25)}$& $\fe_{7(-25)}$ & $3$\\
\hline
\end{tabular}
\caption{$\fg$ simple with complex structure}\label{tab:ClassificationCplx}
\end{table}

\begin{table}
\centering
\begin{tabular}{|c|c|c|c|c|c|}
\hline
& $\fg$ & $\fg^c$&$\fh$& $\rank H/L$\\
\hline
A I & $\sl (p+q,\R)$& $\su(p,q)$& $\so (p,q)$ & $\min\{p,q\}$\\
A II & $\su^*(2(p+q))$&$\su(2p,2q)$& $\sp (p,q)$ & $\min\{p,q\}$\\
A III & $\su (n,n)$& $\su (n,n)$& $\sl (n,\C)\times \R$ & $n$ \\
BD Ia & $\so (n,n)$& $\so^*(2n)$& $\so (n,\C)$ & $\lfloor n/2\rfloor$ \\
BD Ib & $\so (p+1,q+1)$ & $\so (p+q,2)$&$\so (p,1)\times \so (1,q)$ & $2$\\
BD Ic & $\so (n+1,1)$ & $\so(n,2)$&$\so (n,1)$ & $1$\\
C I & $\sp (n,\R)$& $\sp (n,\R)$& $\sl (n,\R)\times \R$ & $n$\\
C II & $\sp (n,n)$& $\sp (2n,\R)$& $\sp (n,\C)$& $n$\\
D III & $\so^*(4n)$& $\so^*(4n)$& $\su^*(2n)\times \R$ & $n$\\
E I & $\fe_{6(6)}$&$\fe_{6(-14)}$&$\sp (2,2)$ & $2$\\
E IV & $\fe_{6(-26)}$&$\fe_{6(-14)}$ & $\ff_{4(-20)}$ & $1$\\
E V & $\fe_{7(7)}$&$\fe_{7(-25)}$&$\su^*(8)$ & $3$\\
E VII & $\fe_{7(-25)}$&$\fe_{7(-25)}$&$\fe_{6(-26)}\times\R$ & $3$\\
\hline
\end{tabular}
\caption{$\fg$ simple without complex structure}\label{tab:ClassificationReal}
\end{table}

\begin{Ex} Let $\K\in\{\R,\C, \H\}$ and $G=\SL(p+q,\K)$ with $p,q\geq1$. If $\K=\H$ this means that $G=\SU^*(2(p+q))$. We choose the maximal compact subgroups
$K$ of $G$ given by $\SO(p+q)$, $\SU(p+q)$ and $\Sp(p+q)$, respectively. Let $\cB=\Gr_p(\K^{p+q})$ be the space of all
$p$-dimensional $\K$-subspaces of $\K^{p+q}$. In the case $\K = \H$ we let the vector space multiplication act on the right and $G$ act on the left.
The group $G$ acts transitively on $\cB$ by $g\cdot b=\{g(v)\mid v\in b\}$. In fact, the maximal compact subgroup $K$ already acts transitively and
$\cB\simeq K/L$ is a symmetric space, where $L$ is the stabilizer of
\[b_0= \K e_1\oplus \cdots \oplus \K e_p\]
with $(e_j)$ denoting the standard basis of $\K^{p+q}$. The stabilizer of $b_0$ in $G$ is the maximal parabolic subgroup $\Pmax=MAN=G_0N$ with
\[ G_0 = \left\{\left. \begin{pmatrix} a & 0 \\ 0 &b\end{pmatrix}\, \right| a\in \GL (p,\K), b\in \GL (q,\K), \det a \det b=1\right\} \]
and
\[N= \left\{\left. n_X=\begin{pmatrix} \rI_p & X\\ 0 & \rI_q\end{pmatrix}\, \right|\, X\in M_{p\times q}(\K )\right\}\simeq M_{p\times q}(\K ). \]
In particular it follows that $N$ is abelian, hence $\cB$ is a symmetric $R$-space with grading element
\[X_0=\begin{pmatrix} \frac{q}{p+q}\rI_p & 0 \\ 0 & -\frac{p}{p+q}\rI_q\end{pmatrix}\, . \]

Define $\bn_X =(n_{X^t})^t$, $X\in M_{q\times p}(\K)$. Then $\overline N=\{\bn_X\mid X\in M_{q\times p}(\K)\}=N^t$ and $\theta (n_X)=\bn_{-X^*}$, where $X^*=\overline{X^t}$ with respect to the standard conjugation of $\K$. Write $\K^{p+q}=\K^p\times \K^q$
and write accordingly the elements of $\K^{p+q}$ as $\bx = (\bx_p,\bx_q)$. Then
\[\bn_X\cdot b =\{(\bx_p,X\bx_p+\bx_q)\mid (\bx_p,\bx_q)\in b\}\, .\]
In particular 
\[b_X := \bn_X\cdot b_0=\mathrm{Graph}(X)=\{(\bx_p, X\bx_p)\mid \bx_p\in \K^p\}\]
and
\[\overline N\cdot b_0 = \{b_X \mid X\in M_{q\times p}(\K)\}=\{b\in\cB\mid \pr_{\K^p}(b)=\K^p\}\]
where $\pr_{\K^p} : \K^{p+q}=\K^p\oplus\K^q\to \K^p$ is the natural projection.
Given $b\in\cB$ with $\pr_{\K^p}(b)=\K^p$ the matrix $X$, viewed as a linear map $\K^p\to \K^q$, such that $b=b_X$ can be recovered from
$b$ by $X=\pr_{\K^q}\circ (\pr_{\K^p}|_b)^{-1}$ where one uses that $\pr_{\K^p}|_b:b\to \K^p$ is a linear isomorphism. 

Identifying $\overline N\cdot b_0\simeq\overline N$, then the almost everywhere defined $G$ action is given by 
\[g\cdot X= (c+dX)(a+bX)^{-1}, \qquad X\in M_{q\times p}(\K),\,g=\begin{pmatrix} a & b\\ c & d\end{pmatrix}\, .\]
We note that this unusual actions comes from our choice of $X_0$. Replacing $X_0$ by $-X_0$ would interchange the
role of $N$ and $\bN$ and lead to the more commonly used action $g\cdot X=(aX+b)(cX+d)^{-1}$, where $g$ is as above and
$X\in M_{p\times q}(\K )$.
 
The involution $\wtt $ is given by conjugation with
$$ \rI_{p,q}:=\begin{pmatrix} \rI_p&0\\ 0 & -\rI_q\end{pmatrix}. $$
The corresponding non-compactly causal involution is 
$\tau =\theta\wtt$ and it corresponds to the following symmetric pairs
$(\fg,\fh)$:
$$ (\sl (p+q,\R), \so (p,q)), \quad (\sl (p+q,\C), \su (p,q)), \quad (\su^*(2(p+q)),\sp (p,q)) $$
for $\K=\R,\C,\H$, respectively. The corresponding Hermitian symmetric pairs $(\fg^c,\fk^c)$ are 
\begin{equation*}
 (\su (p,q),\fs (\fu (p)\times \fu (q))), \quad (\su (p,q)\times \su (p,q), \fs (\fu (p)\times \fu(q))\times \fs(\fu (p)\times \fu(q))),
\end{equation*}
\begin{equation*}
 (\su (2p,2q),\fs(\fu(2p)\times \fu(2q))),
\end{equation*}
respectively, where for $\K=\C$ we have $G^c/K^c =\SU(p,q)/\rS (\rU (p)\times \rU( q))\times \overline{\SU(p,q)/\rS (\rU (p)\times \rU( q)})$, the bar indicating the opposite complex structure.

The spaces $G/\Pmax\simeq\Gr_p(\K^{p+q})$ and $G/\bP\simeq\Gr_q(\K^{p+q})$ are isomorphic as manifolds and $K$-spaces. The
isomorphism is given by $b\mapsto b^\perp$, where the orthogonal complement is taken with respect to the $K$-invariant inner product on $\K^{p+q}$. On the group level this isomorphism corresponds
to $g\mapsto \theta (g)$. On $\cB=\Gr_p(\K^{p+q})$ the involution $\wtt$ corresponds to $\wtt (b)=\rI_{p.q} b =\{(\bx_p,-\bx_q)\mid\bx=(\bx_p,\bx_q)\in b\}$. Hence $\tau (b)=(\rI_{p,q}b)^\perp$.
\end{Ex}

\section{Principal series representations and intertwining operators}\label{Se:Principal}

In this section we recall some basic facts about degenerated principal series representations induced from the maximal parabolic subgroup $\Pmax$. We then introduce the standard intertwining operators and the Berezin transform. The material is mostly a simple generalization of \cite{OP12} to symmetric $R$-space. We therefore often refer to \cite{OP12} for references.

\subsection{Degenerate Principal Series Representations}

Define $\rho \in \fg_0^*$ by $\rho (X):=\frac{1}{2}\Tr (\ad (X)|_{\fn})$. Then $\rho |_{\fm}=0$ and we view $\rho$ as an element in $\fa^*$. If
$X=rX_0$ then $\rho (X)=r \frac{\dim \fn}{2}$. For $g\in G$, $b=k\cdot b_0\in\cB$, and $\lambda\in\fa_\C^*$ we write
\[j_\lambda (g,b):=a(gk)^{-\lambda - \rho}\quad \text{and}\quad j(g,b):= j_\rho (g,b)= a(gk)^{-2\rho}\, .\]

For $\lambda\in\fa_\C^*$ let $\cH_\lambda$ be the Hilbert space of measurable functions $f: G\to \C$ such that 
\begin{enumerate}
\item $f(x man)=a^{-\lambda -\rho} f(x) $ for all $x\in G$ and $man\in MAN$,
\item $\int_K|f(k)|^2\, dk<\infty$.
\end{enumerate}
Then define a representation $\pi_\lambda$ of $G$ acting on $\cH_\lambda$ by
\[\pi_\lambda (g)f(x) := f (g^{-1}x)\, .\] 

Restricting to $K$ and using that $f|_K$ is right $L$-invariant
it follows that $\cH_\lambda \simeq L^2(\cB)$ and that $\pi_\lambda$ acting on $L^2(\cB )$  is given by
\[\pi_\lambda (g)f(b)=j_\lambda (g^{-1},b)f(g^{-1}\cdot b)\, .\]
From this expression it is easy to see that $\pi_\lambda(G)$ leaves $C^\infty(\cB)$ invariant. Note that in the language of parabolically induced representations we have
$$ (\pi_\lambda,L^2(\cB)) \simeq \Ind_{\Pmax}^G(\1\otimes e^\lambda\otimes\1), $$
where the induction is normalized.
 
We recall the following well-known fact which follows from the integral formula
\begin{multline}
 \int_{\cB} f(g\cdot b)j(g,b)\, db=\int_{\cB} f(b)\, db=\int_{\bN} f(\bn\cdot b_0)a(\bn)^{-2\rho}\, d\bn\\
 f\in L^1(\cB),\, g\in G\,.\label{eq:IntFormulaK}
\end{multline}

\begin{Thm}\label{th:dual}
Let $f,h\in L^2(\cB)$ and $g\in G$, then
\[\ip{\pi_\lambda (g)f}{h}_{L^2(\cB)}=\ip{f}{\pi_{-\overline \lambda }(g^{-1})h}_{L^2(\cB)}\, .\]
In particular, $(\pi_\lambda,L^2(\cB))$ is unitary if and only if $\lambda\in i\fa^*$.
\end{Thm}

\begin{Cor}\label{co:InterComp}
Let $\lambda\in\fa^*$ and assume that
$$ A:(\pi_\lambda,C^\infty(\cB))\to(\pi_{-\lambda},L^2(\cB)) $$
is a $G$-intertwining operator. Then the Hermitian form
$$ (f,h)\mapsto\ip{A(\lambda )f}{h}_{L^2(\cB)} $$
on $C^\infty(\cB)$ is $G$-invariant.
\end{Cor}

We also have:
\begin{Thm}[{see \cite[Lemma 5.3]{VW90}}] There exists an open dense subset $U \subset \fa_\C^*$ of full measure such that
$\pi_\lambda$ is irreducible for $\lambda\in U$.
\end{Thm}

We can also realize $\pi_\lambda $ on functions on $\bN$ by restriction. The formula for the representation is then
\begin{align*}
\pi_\lambda (g)f(\bn)&=f(g^{-1}\bn)=f( \overline\nu (g^{-1}\bn)\mu (g^{-1}\bn)\alpha (g^{-1}\bn) \nu (g^{-1}\bn))\\
&=\alpha (g^{-1}\bn )^{-\lambda - \rho}f(g^{-1}\cdot \bn)\, .
\end{align*}
By \eqref{eq:IntFormulaK} we further have
\begin{align*}
 \int_{\bN} |f (\bn)|^2 a(\bn)^{2\Re \lambda}\, d\bn &=\int_{\bN } |f(k (\bn))|^2 a(\bn)^{-2\rho}\, d\bn =\int_K |f(k)|^2\, dk\, .
\end{align*}
The restriction from $G$ to $\bN$ therefore defines a unitary isomorphism $\cH_\lambda\simeq L^2(\bN, a(\bn )^{2\Re \lambda}d\bn)$. In particular,
$\cH_\lambda \simeq L^2(\bN)$ if $\lambda\in i\fa^*$. The corresponding unitary isomorphism $L^2 (\cB )\simeq L^2 (\bN, a(\bn )^{2\Re \lambda }d\bn)$
is given by $f\mapsto f_\lambda$, where
\begin{equation}\label{eq:f-lambda}
f_\lambda (\bn) : =a(\bn )^{-\lambda -\rho} f(\bn \cdot b_0)\, .
\end{equation}

\subsection{The intertwining operators}\label{sec:IntertwiningOperators}
 
In the induced picture \textit{the standard intertwining operator} $J(\lambda)$  is formally given by
\[J(\lambda)f(x)=\int_{\overline N} f (x\overline n)\, d\overline n, \qquad x\in G.\]
Since it is easier to discuss $J(\lambda)$ in the compact picture, we first find an expression for it as an operator acting functions on $\cB$. For this let\begin{equation}\label{de:KerAlpha}
 \alpha_\lambda:\cB\times\cB\to\C, \quad \alpha_\lambda(k_1\cdot b_0,k_2\cdot b_0):=\alpha (k_1^{-1}k_2)^{\lambda-\rho}\, .
\end{equation}
Applying $\theta$ to both sides and taking the inverse, it follows that $\alpha_\lambda $ is symmetric, i.e. $\alpha_\lambda (a , b)=\alpha_\lambda (b ,a )$.
Then an easy computation using the integral formula \eqref{eq:IntFormulaK} shows that formally
\begin{equation}
 J(\lambda)f(x) = \int_{\cB} \alpha_\lambda(x,y)f(y)\,dy, \qquad x\in\cB.\label{eq:IntertwinerIntegralOperator}
\end{equation}
The following statement now makes the construction of the intertwining operators rigorous:

\begin{Thm}[{see \cite{VW90}}]\label{thm:MeromorphicContinuationIntertwiningOperators}
\begin{enumerate}
\item There exists $c\in\R$ such that the integral in \eqref{eq:IntertwinerIntegralOperator} converges for all $\lambda\in\fa_\C^*$ with $\Re (\lambda (X_0))>c$ and $f\in L^2(\cB)$. This constructs an intertwining operator $J(\lambda):(\pi_\lambda,L^2(\cB))\to (\pi_{-\lambda}^\theta,L^2(\cB))$ where $\pi_{-\lambda}^\theta=\pi_{-\lambda}\circ\theta$.
\item For fixed $f\in C^\infty(\cB)$ the function $\lambda\mapsto J(\lambda )f$ extends to a meromorphic function on $\fa^*_\C$ with values in $C^\infty(\cB)$.
\end{enumerate}
\end{Thm}

We now describe the spectrum of the intertwining operator $J(\lambda)$, i.e. its action on the $K$-types of $\pi_\lambda$. For this we first introduce some notation. Denote by $\wKL$ the irreducible unitary $L$-spherical representations $(\delta,V_\delta )$ of $K$. As $L$ is a symmetric subgroup it follows that
$\dim V_\delta^L=1$ for $\delta\in \wKL$.
We fix once and for all an $L$-fixed vector $e_\delta\in V_\delta$ with
$\|e_\delta \|=1$. Then we get a $K$-equivariant isometric embedding
\[\Phi_\delta : V_\delta \hookrightarrow L^2(\cB ), \quad \Phi_\delta (v)(k\cdot b_0) : =(\dim V_\delta )^{1/2} \ip{v}{\pi_\delta (k)e_\delta }\, .\]
We let 
\[L_\delta^2(\cB ):=\Im \Phi_\delta\, .\]

As $\cB$ is a symmetric space it follows that
\begin{equation}\label{de:etaDelta}
L^2(\cB)\simeq_K \bigoplus_{\delta\in \wKL} L^2_\delta (\cB)\simeq_K \bigoplus_{\delta\in \wKL}V_\delta
\end{equation}
where each of the representations $\delta\in \wKL$ occurs with multiplicity one.

The highest weights of the representations in $\wKL$ are given by the Cartan--Helgason-Theorem.
Fix a maximal abelian subspace $\fb\subseteq \fk\cap \fq$ and denote by $\Sigma\subseteq i\fb^*$ the (restricted) roots of $\fk_\C$ with respect to $\fb_\C$. Fix a
positive system $\Sigma^+$ in $\Sigma$ and let
\[\Lambda^+:=\{\mu \in i\fb^* \mid (\forall \alpha \in \Sigma^+)\,\, \frac{\langle \mu ,\alpha \rangle}{\langle \alpha ,\alpha \rangle }\in \N_0\}\, .\]
Then, according to \cite[p.\ 535]{H00}, the map $\pi \mapsto \,(\textrm{highest weight of $\pi$})$ defines an injective map of $\wKL$ into $\Lambda^+$.
This map is bijective if and only if $K$ is simply connected. In general $\wKL$ is isomorphic to a sublattice $\LB$ of $\Lambda^+$. For
$\mu\in\LB$ we denote by $\delta_\mu$ the corresponding spherical representation. We write $V_\mu$, $L_\mu^2(\cB)$ etc. for $V_{\delta_\mu}$, $L_{\delta_\mu}^2(\cB)$ etc.

\begin{Thm}\label{ThmJlambda}
For each $\mu\in\LB$ there exists a meromorphic function $\eta_\mu : \fa^*_\C\to \C$ such that
\[J(\lambda )|_{L^2_\mu (\cB)}=\eta_\mu (\lambda )\, {\id}_{L^2_\mu (\cB)}\, .\]
Moreover, for $\mu=0$ the function $\eta_0(\lambda)$ is given by
\[\eta_0(\lambda )=\int_{\overline N} a(\overline n)^{-\lambda -\rho}\, d\overline n \qquad (\Re (\lambda (X_0))>c)\]
and we have
\[J(-\lambda )\circ J(\lambda )=\eta_0(-\lambda )\eta_0(\lambda )\id\, .\]
\end{Thm}

\begin{proof}
The proofs are the same as in \cite[Theorem 2.6 and Lemma 3.1]{OP12}. We point out that the first statement follows from the multiplicity one decomposition in \eqref{de:etaDelta} and the second statement follows from the fact that $\pi_\lambda$ is irreducible for $\lambda$ in an open dense subset of $\fa_\C^*$.
\end{proof}

The explicit form of the functions $\eta_\mu(\lambda)$ was determined in \cite{OZ95,S93} for $G$ Hermitian, in \cite{S95,Z95} for $G$ non-Hermitian and $\Pmax$ and $\bP$ conjugate, in \cite{OP12} for the Grassmannians $\cB=\Gr_p(\K^{p+q})$, and in \cite{MS14} for the remaining cases.

\subsection{The complementary series}

We identify $\fa^*_\C\simeq\C$ by $\lambda \mapsto 2\lambda (X_0)$. In some cases there exists $R>0$ such that the representations $(\pi_\lambda,C^\infty(\cB))$ are irreducible and unitarizable for $\lambda\in(-R,R)$. Let $R$ be maximal with this property and put $R=0$ if such an interval does not exist.

In case $R>0$, the maximal parabolic subgroup $\Pmax$ and its opposite parabolic $\bP$ are conjugate. More precisely, there exists $w_0\in N_K(\fa)$ such that $\Ad(w_0)|_{\fa }=-1$. Then $w_0Nw_0^{-1}=\bN$ and hence $w_0\Pmax w_0^{-1}=\bP$. Define
$$ A(\lambda ):C^\infty(\cB)\to L^2(\cB), \quad A(\lambda)f(x) = J(\lambda )f(xw_0), $$
then $A(\lambda)$ intertwines $\pi_\lambda$ and $\pi_{-\lambda}$ and therefore, by Corollary~\ref{co:InterComp} the Hermitian form $(f,h)\mapsto\ip{A(\lambda)f}{h}_{L^2(\cB)}$ on $(\pi_\lambda,C^\infty\cB))$ is $G$-invariant for $\lambda \in \fa^*$. This form is positive definite if and only if $\lambda\in(-R,R)$ and in this case it defines a $G$-invariant inner product on $\pi_\lambda$, turning it into an irreducible unitary representation. These representations are called \textit{(degenerate) complementary series}.

The constants $R\geq0$ were obtained for all symmetric $R$-spaces in \cite{MS14,OZ95,S93,S95,Z95} and we summarize the results in Table~\ref{tab:ComplementarySeriesInterval}.

\begin{table}
\centering
\begin{tabular}[t]{|c|c|}
\hline
& $R$\\
\hline
A & $\begin{cases}2p&\mbox{if $p=q$,}\\0&\mbox{if $p\neq q$,}\end{cases}$\\
BD & $4$\\
C & $2n$\\
D & $\begin{cases}n&\mbox{if $n$ is even,}\\0&\mbox{if $n$ is odd,}\end{cases}$\\
E$_6$ & $0$\\
E$_7$ & $6$\\
A I & $\begin{cases}p&\mbox{if $p=q$,}\\0&\mbox{if $p\neq q$,}\end{cases}$\\
A II & $\begin{cases}p&\mbox{if $p=q$,}\\0&\mbox{if $p\neq q$.}\end{cases}$\\
A III & $\begin{cases}n&\mbox{if $n$ is odd}\\0&\mbox{if $n$ is even.}\end{cases}$\\
\hline
\end{tabular}
\begin{tabular}[t]{|c|c|}
\hline
& $R$\\
\hline
BD Ia & $\begin{cases}n&\mbox{if $n$ is even,}\\0&\mbox{if $n$ is odd,}\end{cases}$\\
BD Ib & $\begin{cases}0&\mbox{if $p-q\equiv2$ mod $4$,}\\1&\mbox{if $p-q\equiv1,3$ mod $4$,}\\2&\mbox{if $p-q\equiv0$ mod $4$,}\end{cases}$\\
BD Ic & $\so (n+1,1)$\\
C I & $\begin{cases}n/2&\mbox{if $n$ is even,}\\0&\mbox{if $n$ is odd.}\end{cases}$\\
C II & $3n$\\
D III & $n$\\
E I & $0$\\
E IV & $0$\\
E V & $3$\\
E VII & $3$\\
\hline
\end{tabular}
\caption{The complementary series interval $(-R,R)$}\label{tab:ComplementarySeriesInterval}
\end{table}

\begin{Ex}[The $\cos^\lambda$ transform] \label{ex:CosLambdaTransforms}
The intertwining operator $J(\lambda)$ in Section~\ref{sec:IntertwiningOperators} has a particularly nice interpretation for the Grassmainan $\cB=\Gr_p(\K^{p+q})$ (see \cite{OP12} for details). For simplicity we assume $\K=\R$ and $p\leq q$. We identify $\fa_\C^*\simeq\C$ such that $\rho=(p+q)/2$. We note that this normalization is different from the one used above, but more convenient in this particular example. Let $b,c\in\cB$ be $p$-planes in $\R^{n+1}$ and denote by $\pr_c$ the orthogonal
projection onto $c$. Choose any convex body $E\subset b$ of volume $1$ with $0\in E$ and define $|\mathrm{Cos}(b, c)|$ to be the volume of $\pr_c(E)$. Then we have (see \cite[Thm. 4.1]{OP12})
\[\alpha_\lambda ( b,c)=|\Cos (b,c)|^{\lambda - \rho}\,.\]
In particular, $|\mathrm{Cos}(b,c)|$ is independent of the chosen convex body $E$. Further, we obtain
\begin{equation}\label{eq:cosL}
J(\lambda )f(b)=\int_\cB |\Cos (b,c)|^{\lambda-\rho}f(c)\, dc\, .
\end{equation}

If $p=1$ and $u,v\in \rS^n$ determine the lines $b =\R u,c=\R v$, then $|\Cos (b,c)| =|\ip{u}{v}|=|\cos (\measuredangle (u,v))|$. Lifting $f\in C^(\cB)$ to an even function on the sphere we have
\begin{equation}\label{eq:cosL1}
J(\lambda )f(u)=\int_{\rS^n} |\ip {u}{v}|^{\lambda -\rho}f(v)\, dv=\int_{\rS^n} |\cos \measuredangle (u,v)|^{\lambda -\rho}f(v)\, dv\, .
\end{equation}
This is the motivation for calling the transform (\ref{eq:cosL}) the $\cos^\lambda$-transform. It is then denoted by $\cC^\lambda$ or $\cC^\lambda_{p,q}$. We also note that (up to a constant) the
residue at $\lambda -\rho = -1$ is the Funk--Radon transform 
\[Ff(u ) =\int_{\ip{u}{v}=0} f(v)\, dv\, .\]

The spectrum of the $\cos^\lambda$-transform was calculated in \cite{OP12}. We refer to \cite{OP12,OPR13} for extended references and the history, but only recall
the spectrum for the sphere, to avoid having to introduce too much notation that will not be used elsewhere. The irreducible representations in the decomposition of the sphere are given by the harmonic polynomials of degree $m=0, 1, \ldots $. Only the even degrees occur for the projective space $\cB$ and the corresponding eigenvalues are
\[\eta_{2m}(\lambda )=(-1)^m\frac{\Gamma ((n+1)/2)}{\Gamma (1/2)}
\, \frac{\Gamma\left(\frac{1}{2}(\lambda - \rho +1)\right)\Gamma\left(\frac{1}{2}(-\lambda +\rho)+m\right)}{\Gamma \left(\frac{1}{2}(-\lambda +\rho)\right)\Gamma\left(\frac{1}{2}(\lambda + \rho ) +m\right)}\, .\]

There exists an element $w_0\in N_K(\fa )$ such that $\Ad(w_0)|_{\fa}=-\id_{\fa}$ if and only if $p=q$. Here the intertwining operator $A(\lambda)$ has a
simple geometric interpretation. It is given by
\[A(\lambda )f(b)= \cC^\lambda f(b^\perp )\, .\]
This operator is known under the name $\sin^\lambda$-transform and denoted by
$\cS^\lambda$, see \cite{R13} for generalizations and further discussion. The $K$-spectrum of $\cS^\lambda$ for all Grassmanians was calculated
in
\cite[Lem. 6.3]{OP12}. For the real case the formula reduces to 
\[\cS^\lambda|_{L^2_{2m}(\cB)}=(-1)^m\eta_{2m}(\lambda)\cdot\id_{L^2_{2m}(\cB)}\, .\]
\end{Ex}

\section{The Berezin form}\label{Se:Berezin}

In the last section we saw how to construct a meromorphic family of $G$-invariant Hermitian forms on $C^\infty(\cB)$ in case there exists an element $w_0\in N_K(\fa)$ acting by $-1$ on $\fa$. However, in general such an element does not exists. For instance, in Example~\ref{ex:CosLambdaTransforms} we saw that for the Grassmannian $\cB=\Gr_p(\K^{p+q})$ there exists an element $w_0$ as above if and only if $p=q$. In this section we introduce the Berezin kernel $\beta_\lambda$ which allow us to define a meromorphic family of $H$-invariant
Hermitian forms on $C^\infty(\cB)$. The construction is motivated by the work of Hille~\cite{H99}, see also \cite{vDH97a,vDH97b,vDM98,vDM99, vDP99,FP05} for related work. In fact, the Berezin form we introduce is a special instances of Hille's Berezin form. In our situation $G/H$ is a non-compactly causal symmetric space and we only consider functions on $\cB$ and leave out the case of vector bundles. Our special context allows us to employ some tools specific to this situation and to simplify some of the proofs.

\subsection{The Berezin kernel}

For a function $f$ on $G$ or $\cB$ we define $\wtt_* f=f\circ\wtt$.

\begin{Def}
For $\l \in \mathfrak a_\C^*$ the \emph{Berezin operator} $B(\l)$ is the linear operator on $C^\infty(\mathcal B)$ defined by 
\begin{equation}\label{eq:BerezinOp}
B(\l)=\wt\tau_* \circ J(\l)
\end{equation}
\end{Def}

The Berezin operator $B(\lambda)$ is an integral operator
\[B(\lambda)f(x) = \int_{\cB} \beta_\lambda (x,y) f (y)\, dy\, .\] 
and we call its integral kernel $\beta_\lambda:\cB\times\cB\to\C$ the \textit{Berezin-kernel}. It follows from \eqref{eq:IntertwinerIntegralOperator} and the fact that $\wtt(k)=\tau(k)$ ($k\in K$) that the Berezin kernel is given by
\begin{equation}
\beta_\l(x,y)=\alpha_\lambda (\wtt(x),y) =\alpha\left(\tau(h)^{-1}k\right)^{\l-\rho}\label{eq:BerezinKernel}
\end{equation}
for $x=h\cdot b_0,y=k\cdot b_0\in\cB$.

The canonical Hermitian form $\innerl{\cdot}{\cdot}$ on $C^\infty(\cB)$ associated with $B(\l)$ is defined by  
\begin{equation} \label{eq:canonicalHermitian}
\ip {f}{h}_\lambda:=\ip{B(\lambda)f}{h}_{L^2(\cB)} =\int_{\cB}\int_{\cB}\beta_\lambda (x,y)f(x)\overline{h(y)}\,dx\,dy
\end{equation}
and called the \textit{Berezin form}. For $\Re(\lambda(X_0))>c$ it is given by the convergent integral and extended by meromorphic continuation to $\lambda\in\fa_\C^*$. More precisely, for fixed $f,h\in C^\infty(\cB)$ the expression $\ip{f}{h}_\lambda$ is meromorphic in $\lambda$. To show that the Berezin form is in fact $H$-invariant we need the following lemma:
 
\begin{Lemma} \label{lemma:wttau-pil}
For all $\l \in \mathfrak a_\C^*$ we have as operators on $L^2(\mathcal B)$:
\begin{equation}
\wt\tau_* \circ \pi^\theta_{\l}(g)=\pi_\l(\tau(g)) \circ \wt\tau_*\,.
\end{equation}
\end{Lemma}

\begin{proof}
We note first that $a (\wtt(g))=a(g)$ ($g\in G$) as $\wtt(K)=K$, $\wtt(MN)=MN$ and $\wtt|_\fa=\id_\fa$ by Lemma~\ref{lem:InvolutionsOnPmax}. Hence, $j_\lambda(\wtt(g),\wtt(b))=j_\lambda(g,b)$ for $g\in G$, $b\in\cB$. By the same argument $\wtt(g\cdot b)=\wtt(g)\cdot\wtt(b)$ for $g\in G$, $b\in\cB$. Let $f \in C^\infty(\mathcal B)$, then for all $g\in G$ and $b\in\cB$ we have
\begin{align*}
(\wt\tau_*\circ\pi^\theta_{\l}(g))f(b) &= (\pi^\theta_{\l}(g)f)(\wtt(b))\\ 
&=j_\lambda(\theta(g)^{-1},\wtt(b)) f\big(\theta(g)^{-1}\cdot\wtt(b)\big)\\
&=j_\lambda(\tau(g)^{-1},b) f\big(\wtt_*(\tau(g)^{-1}\cdot b)\big)\\
&=(\pi_\l(\tau(g))\circ\wtt_*)f(b)\,.
\end{align*}
Since $f$ was arbitrary, this shows the claim.
\end{proof}
  
\begin{Prop}\label{prop 4.3}
Then for all $f, h \in C^\infty(\mathcal B)$ we have, as an identity of meromorphic functions of $\l\in\fa_\C^*$:
\begin{equation}
\innerl{\pi_\l(g)f}{h}=\innerl{f}{\pi_{\overline{\l}}(\tau(g)^{-1})h}\,.
\end{equation}
In particular, for $\lambda\in\fa^*$ the Berezin form $\innerl{\cdot}{\cdot}$ is $\pi_\l(H)$-invariant.
\end{Prop}

\begin{proof}
The proof is similar to the proof of \cite[Proposition 3.1.4~(i)]{H99}. First of all, it is sufficient to prove the identity for $\Re(\lambda(X_0))>c$ so that the integral defining $B(\lambda )$ converges absolutely, then the general statement follows by meromorphic continuation. By Theorem~\ref{th:dual} and \ref{thm:MeromorphicContinuationIntertwiningOperators} and Lemma~\ref{lemma:wttau-pil} we have
\begin{align*}
\ip{\pi_\l(g)f}{h}_\lambda&=\ip{(\wtt_*\circ J(\l)\circ\pi_{\l}(g))f}{h}_{L^2(\cB)} \\
&=\ip{(\wt\tau_*\circ\pi^\theta_{-\l}(g)\circ J(\l))f}{h}_{L^2(\cB)}\\
&=\ip{(\pi_{-\l}(\tau(g))\circ\wtt_*\circ J(\l))f}{h}_{L^2(\cB)}\\
&=\ip{(\wt\tau_*\circ J(\l))f}{\pi_{\overline\l}(\tau(g)^{-1})h}_{L^2(\cB)}\\
&=\ip{f}{\pi_{\overline\l}(\tau(g)^{-1})h}_\lambda
\end{align*}
and the proof is complete.
\end{proof}

\subsection{The non-compact picture}\label{sec:BerezinNonCptPicture}

We finally express of the Berezin form in the non-compact picture. For this we introduce the kernel
\begin{equation}\label{def:kappaL}
 \kappa_\lambda:\bN\times\bN\to\C, \quad \kappa_\lambda (x,y)=\alpha (\tau (x)^{-1}y)^{\lambda -\rho}\, .
\end{equation}
Further, recall the isomorphism $L^2 (\cB )\simeq L^2 (\bN, a(\bn )^{2\Re \lambda }d\bn),\,f\mapsto f_\lambda$ from \eqref{eq:f-lambda}.

\begin{Lemma}\label{le:4.4}
Let $f,h \in C^\infty(\mathcal B)$ and $\lambda\in\fa^*_\C$, then
\[ \langle f,h\rangle_\lambda = \int_{\overline{N}} \int_{\overline{N}} \kappa_\lambda (x,y) f_\lambda (x)  
\overline{h_\lambda(y)} \; dx\, dy \,.\]
\end{Lemma}

\begin{proof}
By \eqref{eq:IntFormulaK} we have
$$ \langle f,h\rangle_\lambda = \int_{\bN}\int_{\bN} \beta_\lambda(x\cdot b_0,y\cdot b_0)f(x\cdot b_0)\overline{h(y\cdot b_0)}a(x)^{-2\rho}a(y)^{-2\rho}\,dx\,dy. $$
Write $x=k(x)m(x)a(x)n(x)$, then $k(x)=xm(x)^{-1}a(x)^{-1}\widetilde{n}(x)$ for some $\widetilde{n}(x)\in N$. With the same notation for $y$ we obtain
\begin{multline*}
 \beta_\lambda(x\cdot b_0,y\cdot b_0) = \alpha(\tau(k(x))^{-1}k(y))^{\lambda-\rho}\\
 = \alpha(\tau(\widetilde{n}(x))^{-1}\tau(a(x))\tau(m(x))\tau(x)^{-1}ym(y)^{-1}a(y)^{-1}\widetilde{n}(y))^{\lambda-\rho}.
\end{multline*}
Now $\tau(N)=\bN$ and the function $\alpha$ is left $M\overline{N}$-invariant and right $MN$-invariant. Further, $\tau(a(x))=a(x)^{-1}$, so that the above expression is equal to
$$ a(x)^{-\lambda+\rho}a(y)^{-\lambda+\rho}\kappa_\lambda(x,y). $$
By the definition of $f_\lambda$ and $h_\lambda$ this gives the desired expression.
\end{proof}

\section{The restriction of the Berezin form to an open $H$-orbit}\label{se:openOrb}

In order to study the restriction of the Berezin form to the open $H$-orbits in $\cB$ we first describe these orbits using roots of $\fg$. It turns out that each $H$-orbits is a symmetric space and we determine the involution explicitly. We illustrate the orbit decomposition with the example of the Grassmanians $\cB=\Gr_p(\R^{p+q})$. Finally, we write the Berezin form as a sum over integrals over the open $H$-orbits. 

\subsection{The open $H$-orbits in $\cB$}\label{sec:OpenHorbits}

We refer to \cite{HO97,K87,NO00,O91} for the discussion about root systems and Weyl groups related to non-compactly causal spaces. Let $\amin$ be a maximal abelian subspace of $\fp$ containing $X_0$. Then $\amin\subset \fz_\fp (X_0)\subseteq \fp\cap \fq$. Denote by $\Sigma$ the set of roots of $\amin$ in $\fg$. Let $\Sigma_0=\{\alpha \in\Sigma \mid \alpha (X_0)=0\}$ and $\Sigma_{\pm}=\{\alpha \in \Sigma\mid \alpha (X_0)=\pm 1\}$, then 
\[ \Sigma_0=\{\alpha \in \Sigma \mid \fg_\alpha \subset \fg_0\}, \qquad \mbox{and} \qquad \Sigma_\pm=\{\alpha \in \Sigma\mid \fg_\alpha \subset \fg_{\pm1}\}.\]
Furthermore
\[ \fn =\bigoplus_{\alpha \in\Sigma_+}\fg_\alpha\quad\text{ and }\quad \overline \fn =\bigoplus_{\alpha \in\Sigma_-}\fg_\alpha. \] 

Let $\W =W(\amin )=N_K(\amin )/Z_K(\amin)$ and $W_{H\cap K}=W_{H\cap K}(\amin )=N_L(\amin )/Z_L(\amin)$. Note that, by the definition of $L$, we have $Z_K(\amin) =Z_L(\amin)$. Then $\W$ is the Weyl group generated by the
reflections $s_\alpha $ ($\alpha\in\Sigma$) and $W_{H\cap K}$ is the Weyl group generated by $s_\alpha$ ($\alpha \in
\Sigma_0$), i.e. $\W=W(\Sigma)$ and $W_{H\cap K}=W(\Sigma_0)$.
We choose a set of positive roots $\Sigma^+\subseteq\Sigma$ such that $\Sigma_+\subseteq \Sigma^+$. Then $\Sigma^+=\Sigma_0^+\dot\cup \Sigma_+$ with $\Sigma_0^+=\Sigma_0\cap\Sigma^+$ a positive system in $\Sigma_0$. We note that $\WH (\Sigma_+)=\Sigma_+$. Let $\fn_0=
\bigoplus_{\alpha \in \Sigma_0^+} \fg_\a \subset \fg_0$, then
\[\nmin =\bigoplus_{\alpha \in \Sigma^+}\fg_\alpha =\fn_0\ltimes \fn. \]
 
 Two roots $\alpha,\beta\in\Sigma^+$ are called \textit{strongly orthogonal} if $\alpha \not\in \R\beta $ and $\alpha \pm \beta\not\in\Sigma$. If $\alpha$ and
 $\beta$ are strongly orthogonal then they are orthogonal. Let $\alpha_1,\ldots ,\alpha_{r}\in \Sigma_+$ be a maximal set of long strongly orthogonal
 roots. For $j=1,\ldots,r$ let $E_j \in \fg_{\alpha_j}$ and $F_j=\tau (E_j)=-\theta (E_j)\in\fg_{-\alpha_j}$ such that with $H_j =[E_j,F_j]$ the
 map
 \[E_j \mapsto \begin{pmatrix} 0 & 1 \\ 0 & 0\end{pmatrix},\quad F_j \mapsto \begin{pmatrix} 0 & 0 \\ 1 & 0\end{pmatrix},\quad
 \text{and} \quad H_j \mapsto \begin{pmatrix} 1 & 0 \\ 0 & -1\end{pmatrix}\]
 is an isomorphism   
$ \R F_j \oplus \R H_j \oplus\R E_j\simeq \sl(2,\R)$
intertwining the involutions $\tau$ and $\theta$ with the involutions on $\sl (2,\R)$ given by conjugation by $E_{12}+E_{21}$ and
the standard Cartan involution $X\mapsto -X^t$. 
Let
$$ s^\prime_{\alpha_j}=\exp \left(\frac{\pi}{2}(E_j-F_j)\right), $$
then $\Ad (s^\prime_{\alpha_j})=s_{\alpha_j}$ is the
Weyl group element corresponding to the reflection in the hyperplane $\alpha_j=0$. 
Note that $s^\prime_{\alpha_i}$ and $s^\prime_{\alpha_j}$ commute because $[E_i,E_j]=[F_i,E_j]
=[E_i,F_j]=[F_i,F_j]=0$. Furthermore, as $\alpha_i(X_0)=1$ and $\alpha_i(H_i)=2$,
\[s_{\alpha_i}(X_0)=X_0-\frac{1}{\ip{\alpha_i}{\alpha_i}}H_i\]
and 
\begin{align*}
 s_{\alpha_i}^2(X_0) &=s_{\alpha_i}\left(X_0-\frac{1}{\ip{\alpha_i}{\alpha_i}}H_i\right)\\
 &=X_0-\frac{1}{\ip{\alpha_i}{\alpha_i}}H_i-\left(\frac{1}{\ip{\alpha_i}{\alpha_i}}H_i-
\frac{2}{{\ip{\alpha_i}{\alpha_i}}}H_i\right)=X_0\,.
\end{align*}
Thus $(s^\prime_{\alpha_i})^2\in L$.

Define
\[ w_j=s_{\alpha_1}\cdots s_{\alpha_j}\in W(\Sigma),\quad j=0,\ldots ,r\,,\]
and the corresponding representatives
\[ \tilde{w}_j=s_{\alpha_1}'\cdots s_{\alpha_j}'\in K,\quad j=0,\ldots ,r\,.\]

We will use the following two standard results:

\begin{Lemma}[{see \cite[Lemma 5.4~(1)]{T79}}]
The set $\{w_0,\dots ,w_r\}$ is a set of representatives of $W(\Sigma_0)\backslash W(\Sigma)/W(\Sigma_0)$, i.e., 
\[W(\Sigma)=\dot\bigcup_{j=0}^{r}  W(\Sigma_0)w_j W(\Sigma_0)\, .\] 
\end{Lemma}

\begin{Lemma}[{see \cite[Lem. 4.3]{O91}}]\label{lem:MaxAbelianInPcapH}
Let $\fa_{\fh}=\bigoplus_{i=1}^{r} \R (E_i+F_i)$, then $\fa_\fh$ is maximal abelian in $\fh\cap \fp$. In particular, $r=\rank(H/L)=\rank(K/L)$.
\end{Lemma}

\begin{Thm}\label{th:double}
The open $H$-orbits in $G/\Pmax$ are $\cO_j=H\cdot\tilde{w}_j\Pmax$, $j=0,\ldots,r$. In particular, the number of open $H$-orbits is $r+1$.
\end{Thm}

\begin{proof}
Let $W_{\Pmax} =N_{\Pmax} (\amin )/Z_{\Pmax} (\amin )$. Then $W_{\Pmax}=W_{H\cap K}$. As $G_0=L(\Pmin\cap G_0)$ it follows that $H\Pmax=H \Pmin$ is open. The claim now follows from \cite[Cor. 16]{R79}. See also the remark on  p. 317 in \cite{M82} and \cite[Lem. 5.4.15]{HO97}.
\end{proof}

\subsection{The stabilizer subgroups}

For each $j=0,\ldots,r$ denote by
$$ H_j = H\cap\tilde{w}_j\Pmax\tilde{w}_j^{-1} $$
the stabilizer of $\tilde{w}_jb_0$ in $H$. Then the open $H$-orbit $H\cdot\tilde{w}_jb_0$ can be identified with the homogeneous space $H/H_j$. We show that $H_j$ is a symmetric subgroup of $H$.

Let $j=0,\ldots,r$ and define an automorphism $\sigma_j$ of $G_\C$ by
$$ \sigma_j(g) := \tilde{w}_j^2\tilde{\tau}(g)\tilde{w}_j^{-2}, \qquad g\in G_\C. $$
Note that $\sigma_0=\tilde{\tau}$.

\begin{Lemma}
$\tilde{w}_j^4=1$ and hence $\sigma_j$ is an involution. Moreover, $\sigma_j$ leaves $H$ invariant.
\end{Lemma}

\begin{proof}
Consider the $\sl_2$-triple $(E_j,F_j,H_j)$ in $\fg_\C$. Since $\SL(2,\C)$ is simply connected there is a unique group homomorphism $\varphi_j:\SL(2,\C)\to G_\C$ such that $d\varphi_j$ maps the standard $\sl_2$-triple $(E,F,H)$ in $\sl(2,\C)$ to $(E_j,F_j,H_j)$. Hence
$$ \tilde{w}_j = \varphi_j\left(\begin{array}{cc}0&1\\-1&0\end{array}\right). $$
This implies that $\tilde{w}_j^4=1$ which proves the first statement. For the second statement note that $\tilde{\tau}$ leaves $H$ and $G^c$ invariant. Further, a short computation shows that
$$ \Ad(\exp(t(E_j-F_j)))X_0 = (X_0-\tfrac{1}{2}H_j) + \tfrac{1}{2}\cos(2t)H_j - \tfrac{1}{2}\sin(2t)(E_j+F_j). $$
In particular, $\Ad(\tilde{w}_j^2)X_0=\Ad(\exp(\pi(E_j-F_j)))X_0=X_0$ so that $\tilde{w}_j^2\in Z_K(X_0)\subseteq H$. Hence $\sigma_j$ leaves $H$ invariant and the proof is complete.
\end{proof}

\begin{Prop}\label{pro:5.5}
For every $j=0,\ldots,r$ we have $H_j=H^{\sigma_j}$, in particular the open $H$-orbits in $\cB$ are symmetric spaces.
\end{Prop}

\begin{proof}
Let $h\in H$. Then
$$ h\in\tilde{w}_j\Pmax\tilde{w}_j^{-1} \quad\Leftrightarrow\quad \tilde{w}_j^{-1}h\tilde{w}_j\in\Pmax. $$
We claim that
$$ \tilde{w}_j^{-1}h\tilde{w}_j\in\Pmax \quad\Leftrightarrow\quad \tilde{w}_j^{-1}h\tilde{w}_j\in G_0. $$
The direction $\Leftarrow$ is clear, so we assume $\tilde{w}_j^{-1}h\tilde{w}_j\in \Pmax$. Applying $\tau$ yields
\begin{equation}
 \tau(\tilde{w}_j^{-1}h\tilde{w}_j) = \tilde{w}_jh\tilde{w}_j^{-1} \in\tau(\Pmax) = \bP.\label{eq:TauOfConjugateH}
\end{equation}
Since $\Ad(\tilde{w}_j^2)X_0=X_0$ we have $\tilde{w}_j^2\bP\tilde{w}_j^{-2}=\bP$. Applying this to \eqref{eq:TauOfConjugateH} yields
$$ \tilde{w}_j^3h\tilde{w}_j^{-3} = \tilde{w}_j^{-1}h\tilde{w}_j \in \bP. $$
This implies that $\tilde{w}_j^{-1}h\tilde{w}_j\in\Pmax\cap\bP=G_0=G^{\tilde{\tau}}$ and we obtain
$$ h\in\tilde{w}_j\Pmax\tilde{w}_j^{-1} \quad\Leftrightarrow\quad \tilde{w}_j^{-1}h\tilde{w}_j\in G_0 \quad\Leftrightarrow\quad \tilde{w}_j^{-1}h\tilde{w}_j=\tilde{\tau}(\tilde{w}_j^{-1}h\tilde{w}_j). $$
But now
$$ \tilde{\tau}(\tilde{w}_j^{-1}h\tilde{w}_j) = \tilde{w}_j\tilde{\tau}(h)\tilde{w}_j^{-1} $$
and therefore
$$ \tilde{w}_j^{-1}h\tilde{w}_j=\tilde{\tau}(\tilde{w}_j^{-1}h\tilde{w}_j) \quad\Leftrightarrow\quad h = \tilde{w}_j^2\tilde{\tau}(h)\tilde{w}_j^{-2}. $$
This shows $H_j=H^{\sigma_j}$.
\end{proof}

\begin{Ex}[The real Grassmannians]
Let $G=\SL(p+q,\R)$ acting on the Grassmannian $\cB=\Gr_p(\R^{p+q})$ of $p$-dimensional subspaces in $\R^{p+q}$. Then
$$ \Pmax = \left\{\left.\begin{pmatrix} A & B \\ 0 & D\end{pmatrix}
\, \right|\, \begin{array}{cc}A\in \GL (p,\R),D\in\GL(q,\R),B\in M_{p\times q}(\R),\\(\det A)(\det D)=1\end{array} \right\} $$
and $H=\SO(p,q)$, the indefinite orthogonal group with respect to the symmetric bilinear form
$$ \omega(x,y) = x_1y_1+\cdots+x_py_p-x_{p+1}y_{p+1}-\cdots-x_{p+q}y_{p+q}. $$
Assume $1\leq p\leq q$, then ${\rm rank}(H/L)=p$ and the $p+1$ open $H$-orbits $\cO_0,\ldots,\cO_p$ in $\cB$ are given by
$$ \cO_j = \{b\in\cB:\omega|_{b\times b}\mbox{ has signature }(p-j,j)\} \qquad (0\leq j\leq p). $$
The elements $b_0,\ldots,b_p\in\cB$ given by
$$ b_j = \R e_1+\cdots+\R e_{p-j}+\R e_{p+1}+\cdots+\R e_{p+j} $$
are representatives of the orbits, i.e. $\cO_j=H\cdot b_j$, $0\leq j\leq p$. The stabilizer of $b_j$ in $H$ is given by $H_j=\upS(\upO(p-j,j)\times\upO(j,q-j))$, so that $\cO_j\simeq H/H_j$. In particular, $H_0=\upS(\upO(p)\times\upO(q))$ is the maximal compact subgroup of $H=\SO(p,q)$ and $\cO_0\simeq H/H_0$ is the corresponding Riemannian symmetric space. Note that for $p=q$ also $H_p$ is the maximal compact subgroup of $H$ and hence both $\cO_0$ and $\cO_p$ are Riemannian symmetric spaces.
\end{Ex}

\subsection{Restricting the Berezin form}

According to \cite[Lemma 1.3]{O87} we have
\begin{equation}
 \int_{\cB}f(b)db =\sum_{j=0}^r \int_{H/H_j}\, f(g\tilde{w}_j \cdot b_0)a(g\tilde{w}_j)^{-2\rho}\, dg\, ,\label{eq:SumOverHOrbits}
\end{equation}
where $dh$ denotes the (suitably normalized) $H$-invariant measure on $H/H_j$. Note that since $H/H_j$ is a symmetric space, invariant measures always exists. This integral formula can be used to rewrite the restriction of the Berezin form to an open $H$-orbit $H\cdot\tilde{w}_jb_0$ in terms of integrals over the symmetric space $H/H_j$. To state the result we define for a function $f$ on $\cB$:
$$ f_{j,\lambda}: H/H_j\to\C, \quad f_{j,\lambda}(g) = a(g\tilde{w}_j)^{-\lambda-\rho}f(g\tilde{w}_j\cdot b_0). $$

\begin{Thm}
Assume that $f,h\in C^\infty(\cB)$ have compact support inside the open $H$-orbit $H\cdot\tilde{w}_jb_0\subseteq\cB$, then
\[\herm{f}{h}_\lambda=\int_{H/H_j}\int_{H/H_j} \alpha (\tilde{w}_jg^{-1}k\tilde{w}_j)^{\lambda-\rho}f_{j,\lambda}(g) \overline{h_{j,\lambda}(k)}
\, dg\, dk\, .\]
\end{Thm}

\begin{proof}
By \eqref{eq:SumOverHOrbits} we have
\begin{multline*}
 \herm{f}{h}_\lambda = \int_{H/H_j}\int_{H/H_j} \beta_\lambda(g\tilde{w}_j\cdot b_0,k\tilde{w}_j\cdot b_0)f(g\tilde{w}_j\cdot b_0)\overline{h(k\tilde{w}_j\cdot b_0)}\\
 a(g\tilde{w}_j)^{-2\rho}a(k\tilde{w}_j)^{-2\rho}\,dg\,dk\, ,
\end{multline*}
and by \eqref{eq:BerezinKernel} the Berezin kernel is given by
$$ \beta_\lambda(g\tilde{w}_j\cdot b_0,k\tilde{w}_j\cdot b_0) = \alpha(\tau(k(g\tilde{w}_j))^{-1}k(k\tilde{w}_j))^{\lambda-\rho}. $$
Write $g\tilde{w}_j=k(g\tilde{w}_j)m(g\tilde{w}_j)a(g\tilde{w}_j)n(g\tilde{w}_j)$, then
$$ k(g\tilde{w}_j)=g\tilde{w}_jm(g\tilde{w}_j)^{-1}a(g\tilde{w}_j)^{-1}\widetilde{n}(g\tilde{w}_j) $$
for some $\widetilde{n}(g\tilde{w}_j)\in N$ and similar for $k\tilde{w}_j$. Note that $\tau(M)=M$, $\tau|_\fa=-\id_\fa$, $\tau(N)=\overline N$ and $\tau(\tilde{w}_j)=\tilde{w}_j^{-1}$. This implies
\[ \alpha(\tau(k(g\tilde{w}_j))^{-1}k(k\tilde{w}_j))^{\lambda-\rho} = \alpha(\tilde{w}_jg^{-1}k\tilde{w}_j)^{\lambda-\rho}a(g\tilde{w}_j)^{-\lambda+\rho}a(k\tilde{w}_j)^{-\lambda+\rho} \]
where we have used that $\tau(g)=g$ since $g\in H$. Now the claim follows by the definition of $f_{j,\lambda}$ and $h_{j,\lambda}$.
\end{proof}

\section{Reflection Positivity}\label{sec:ReflectionPositivity} 
In this section we recall the basic definitions related to \textit{reflection positivity}, formulated so that it fits our setup. For
basic references we point to \cite{JO00,NO14,NO17b}. For other aspects of reflection positivity we
would like to name \cite{JZ17,JJ16,JJ17,JP15a,JF15b,KL83}.

Let $(\cE,\pi)$ be a Casselman--Wallach representation of $G$ on a Fr\'{e}chet space $\cE$ (i.e. $\pi$ is smooth, admissible and of moderate growth). Assume we are given a Hermitian form $(\cdot,\cdot)$ on $\cE$ which is invariant under $\pi\otimes\pi^\theta$, where $\pi^\theta=\pi\circ\theta$.

\begin{Ex}
Let $P=MAN\subseteq G$ be any parabolic subgroup. Then for the corresponding generalized principal series representation $\pi=\pi_\lambda=\Ind_P^G(1\otimes e^\lambda\otimes 1)$ with $\lambda\in\fa^*$ the standard intertwining operator $J(\lambda):\pi_\lambda\to\pi_{-\lambda}^\theta$ can be used to define such a Hermitian form by
$$ (f_1,f_2) := \ip{J(\lambda)f_1}{f_2}_{L^2(K/K\cap M)} = \int_K J(\lambda)f_1(x)\,\overline{f_2(x)}\,dx. $$
This follows similarly as in Theorem~\ref{th:dual} and \ref{thm:MeromorphicContinuationIntertwiningOperators}.
\end{Ex}

Now assume that there exist:
\begin{enumerate}
\item\label{enum:RePoAss1} an isometric involution $\wt\tau_*:\cE\to\cE$ such that
\begin{equation*}
 \wt\tau_*\circ\pi^\theta (g) = \pi(\tau(g))\circ\wt\tau_*,
\end{equation*}
\item\label{enum:RePoAss2} a closed $\pi(H)$-invariant and $d\pi(\fg)$-invariant subspace $\cE_+\subseteq\cE$ such that $(\wt\tau_*f,f)\geq0$ for all $f\in\cE_+$.
\end{enumerate}
We refer to assumption~\ref{enum:RePoAss2} as \textit{reflection positivity}.

Under the above assumptions, we consider on $\cE_+$ the positive semidefinite Hermitian form
$$ \herm{f_1}{f_2} := (\wt\tau_*f_1,f_2). $$
Clearly $\herm{\cdot}{\cdot}$ is $\pi\otimes\pi^\tau$-invariant. Let
\[ \cN :=\{f\in\cE_+\mid \herm{f}{f}=0\}, \]
denote by $\wE$ be the completion of $\cE_+/\cN$ with respect to $\herm{\cdot}{\cdot}$ and by $q:\cE_+\to\wE$ the canonical projection. We
also write $\widehat{f}=q(f)$. For a continuous linear operator $T:\cE\to \cE$ with $T(\cE_+)\subset\cE_+$ and $T(\cN)\subseteq\cN$ we
define $\widehat{T}: \wE\to \wE$ by $T(\widehat{f})=\widehat{T(f)}$. Then $T$ is linear and continuous.

It is clear that $\pi(H)\cN\subseteq\cN$ and we therefore get a unitary representation of $H$ on $\wE$. Further, it follows that $d\pi (\fg )\cN\subseteq\cN$ and therefore
$$ d\pi^c(X+iY) := d\pi(X)+i\,d\pi(Y), \qquad X+iY\in\fh+i\fq=\fg^c $$
defines an infinitesimally unitary representation of $\fg^c$ on $\cE_+/\cN$. The question is whether this representation integrates to a unitary representation of $G^c$, or more generally its universal cover $\widetilde{G}^c$, on $\wE$.

\begin{Ex}
We can take $\pi=\pi_\lambda$ with $\cE=C^\infty(\cB)$, then the involution
$$ \widetilde{\tau}_*: \cE\to\cE, \quad \wtt_*f(b) = f(\wtt(b)) $$
satisfies assumption \ref{enum:RePoAss1}. Further, for any $H$-orbit $\cO_j=H\cdot\tilde{w}_jb_0$ the subspace $\cE_+=C_c^\infty(\cO_j)$ satisfies assumption \ref{enum:RePoAss2} if and only if
$$ \langle f,f\rangle_\lambda = \int_{\cO_j}\int_{\cO_j} \beta_\lambda(x,y)f(x)\overline{f(y)}\,dx\,dy \geq 0 \qquad \forall\,f\in\cE_+. $$
We now determine for each open $H$-orbit $\cO_j$ the parameters $\lambda$ such that reflection positivity holds. We distinguish the two cases of Riemannian and non-Riemannian open $H$-orbits.
\end{Ex}

\section{The Riemannian open $H$-orbits}\label{sec:HighestWeightRepresentations}

We show that on the Riemannian open $H$-orbits the Berezin form is reflection positive if and only if $\lambda$ is contained in the so-called Wallach set which is the union of an unbounded interval with a finite number of discrete points. In this case, the representation $d\pi_\lambda^c$ of $\fg^c$ integrates to an irreducible unitary representation of $\widetilde{G}^c$ on $\wE$ which we identify with a unitary highest representation of scalar type. Most of this material can also be found in \cite{O00}. We further provide an explicit embedding (as an integral operator) of this representation into $L^2(\widetilde{G}^c/\widetilde{H})$, where $\widetilde{H}\subseteq\widetilde{G}^c$ denotes the preimage of $H$ under the covering map $\widetilde{G}^c\to G^c$.

\subsection{The Highest Weight Representations}

Consider the open dense Bruhat cell $\fp^c_-\subseteq G_\C/K^c_\C P^c_+$. Then the orbit of $G^c$ through the base point of $G_\C/K^c_\C P^c_+$ is contained in $\fp^c_-$ and forms a bounded symmetric domain $\cD\subseteq\fp^c_-$. We have $\cD\simeq G^c/K^c$ as $G^c$-spaces. Since most of the representations we construct only live on the universal cover $\widetilde{G}^c$ of $G^c$ we identify $\cD\simeq\widetilde{G}^c/\widetilde{K}^c$.

For $\mu\in\fa_\C^*=(\C Z_0)^*$ denote by $\chi_\mu:\widetilde{K}^c\to\C^\times$ the character whose derived character $d\chi_\mu$ on $\fk^c$ agrees with $\mu$ on $\R Z_0$ and is trivial on $Z_0^\perp\subseteq\fk^c$, the orthogonal complement of $Z_0$ in $\fk^c$ with respect to the Killing form of $\fg$. We consider the kernel function
\begin{equation}\label{eq:Klambda}
 K_\mu(z,w) = \chi_\mu\big(k^c(\exp(-\overline{w})\exp z)\big)^{-1}, \qquad z,w\in\cD,
\end{equation}
where $k^c(g)$ is as in Section~\ref{sec:BoundedDomain}. Here $\overline{w}$ denotes the complex conjugation on $\fg_\C$ with respect to the real form $\fg^c$, then $\overline{w}\in\fp^c_+$. In this notation $\overline{w}=\tau(w)$. Note that $\exp(-\overline{w})\exp z\in P^c_-K^c_\C P^c_+$ for all $z,w\in\cD$, so that the kernel $K_\mu$ is well-defined for all $\mu\in\fa_\C^*$.

\begin{Thm}[{see \cite{B75,VR76,W79}}]\label{thm:WallachSet}
There exists a constant $c>0$ such that the kernel $K_\mu(z,w)$ is positive semidefinite if and only if $\mu(iZ_0)=-\mu(X_0)$ is contained in the so-called Berezin--Wallach set
$$ \cW = (-\infty,-(r-1)c)\cup\{-jc:j=0,\ldots,r-1\}. $$
\end{Thm}

In the case where $K_\mu$ is positive semidefinite, we can form a Hilbert space $\cH^c_\mu$ of holomorphic functions on $\cD$ with reproducing kernel $K_\mu(z,w)$. More precisely, we form the linear span of all functions $z\mapsto K_\mu(z,w)$ ($w\in\cD$) and endow it with the inner product
$$ \inner{K_\mu(\cdot,w_1)}{K_\mu(\cdot,w_2)}_\mu^c = K_\mu(w_2,w_1). $$
Its completion with respect to $\inner{\cdot}{\cdot}_\mu^c$ is a Hilbert space $\cH_\mu^c$ of holomorphic functions on $\cD$ with reproducing kernel $K_\mu(z,w)$. On this Hilbert space there exists an irreducible unitary representation $(\rho_\mu,\cH_\mu^c)$ of $\widetilde{G}^c$ given by
\begin{equation}
 \rho_\mu(g)f(z) = J_\mu(g^{-1},z)^{-1}f(g^{-1}\cdot z),
\end{equation}
where for $g\in\widetilde{G}^c$ and $z\in\cD\subseteq\fp_-^c$ we put
$$ g\cdot z = \log(p_-(g\exp(z))), \qquad \mbox{and} \qquad J_\mu(g,z) = \chi_\mu(g\exp(z)). $$
These representations are highest weight representations of scalar type, and they form the so-called \textit{analytic continuation of the holomorphic discrete series}. We note that for $\mu(iZ_0)\ll0$ the representation $(\rho_\mu,\cH_\mu^c)$ belongs to the holomorphic discrete series and the $\widetilde{G}^c$-invariant inner product on $\cH_\mu^c$ is the $L^2$-inner product
$$ \inner{f}{h}_\mu^c = \int_\cD f(z)\overline{h(z)}\,d\nu_\mu(z), $$
where $d\nu_\mu(z)=K_\mu(z,z)^{-1}\,dz$ and $dz$ denotes a suitably normalized $\widetilde{G}^c$-invariant measure on $\cD$.

\subsection{Positivity of the Berezin form}

Using Theorem~\ref{thm:WallachSet} we now determine for which parameters $\lambda$ the Berezin form $\inner{\cdot}{\cdot}_\lambda$ restricted to the open $H$-orbit $\cO_0=H\cdot b_0\subseteq\cB$ is positive semidefinite. Recall that in the non-compact picture the Berezin form is given by the kernel function $\kappa_\lambda$ on $\overline{N}\times\overline{N}$ (see Section~\ref{sec:BerezinNonCptPicture}). For the following statement we identify $\overline{\fn}\simeq\overline{N}$.

\begin{Lemma}\label{lem:RestrictionOfKmu}
The restriction of $K_\mu$ to $\overline{\fn}\subseteq\fp^c_-$ is equal to $\kappa_\lambda$ for $\lambda=-\mu+\rho$.
\end{Lemma}

\begin{proof}
For $x,y\in\overline{\fn}$ we have
$$ \kappa_\lambda(x,y) = \alpha(\tau(x)^{-1}y)^{\lambda-\rho} = \alpha(\tau(y)^{-1}x)^{\lambda-\rho} $$
since $\tau(N)=\overline{N}$ and $\tau(X_0)=-X_0$. Further, by Lemma~\ref{lem:OpenDenseDecompositionsAgree} we have $k^c(g)=\mu(g)\alpha(g)$ for all $g\in\overline{N}G_0N$ and hence
$$ \chi_\mu(g)^{-1} = \alpha(g)^{\lambda-\rho} $$
for $\lambda=-\mu+\rho$. This shows the claim.
\end{proof}

Now consider the open $H$-orbit $\cO_0=H\cdot b_0$ through the base point $b_0=e\Pmax\in\cB$, then $\cO_0\simeq H/L$ is a Riemannian symmetric space for $H$. Since $MAN=G\cap K_\C^cP^c_+$ and $H=G\cap G^c$, we can view $\cO_0$ as the $H$-orbit through the origin $0$ in the standard bounded realization of $\cD\simeq G^c/K^c$ in $\fp^c_-$. Then $\tau$ induces a conjugation on $\cD$ whose fixed points $\cD_\R=\cD\cap\overline{\fn}\subseteq\cD$ form a totally real submanifold, and we have $\cO_0\simeq\cD_\R$.

From this discussion it follows that $\cO_0\subseteq\overline{N}\cdot b_0$, so that the positivity of $\inner{\cdot}{\cdot}_\l$ on $\cO_0$ can be detected in the non-compact picture. As explained in Section~\ref{sec:BerezinNonCptPicture}, the Berezin form $\inner{\cdot}{\cdot}_\l$ is in the non-compact picture on $\overline{\fn}$ given by the kernel $\kappa_\lambda$.

\begin{Prop}\label{prop7.3}
The restriction of the Berezin form $\inner{\cdot}{\cdot}_\l$ to the Riemannian open orbit $\cO_0$ is positive semidefinite if and only if $K_\mu$ is positive semidefinite for $\mu=-\l+\rho$. This is precisely the case if $(\lambda-\rho)(X_0)$ is contained in the Berezin--Wallach set $\cW$.
\end{Prop}

\begin{proof}
By Lemma~\ref{lem:RestrictionOfKmu} the kernel $\kappa_\lambda$ of $\inner{\cdot}{\cdot}_\l$ is the restriction of $K_\mu$ to the real form $\cD_\R=\cD\cap\overline{\fn}\simeq H\cdot b_0$ of $\cD$. Then the statement is a consequence of \cite[Theorem A.1]{NO14}.
\end{proof}

\begin{Rem}
In the case where $\cD$ is a tube type bounded symmetric domain, there is another Riemannian open $H$-orbit in $\cB$, namely $\cO_r=H\cdot\tilde{w}_rb_0$ where $r=\rank(H/L)$. Note that $\tilde{w}_r^2\in L$ so that multiplication by $\tilde{w}_r$ defines an isomorphism $\cO_0\to\cO_r$. This isomorphism preserves the Berezin form and we obtain that $\inner{\cdot}{\cdot}_\l$ is positive semidefinite on $\cO_r$ if and only if it is positive semidefinite on $\cO_0$, which is the case for $(\lambda-\rho)(X_0)\in\cW$.
\end{Rem}

\subsection{The Intertwining Operator into the Highest Weight Representation}\label{se:7.3}

Now assume that $\lambda\in\fa^*$ such that the Berezin form $\inner{\cdot}{\cdot}_\l$ is positive semidefinite on the open $H$-orbit $\cO_0$. Then for $\cE=C^\infty(\cB)$ and $\cE_+=C_c^\infty(\cO_0)$ the construction in Section~\ref{sec:ReflectionPositivity} yields a pre-Hilbert space $\cE_+/\cN$ on whose completion $\wE$ the group $H$ acts unitarily. Further, $\fg^c$ acts on $\cE_+/\cN$ by infinitesimally unitary operators. We now show that this representation integrates to an irreducible unitary representation of $\widetilde{G}^c$ on $\wE$, and we identify this representation with one of the $(\rho_\mu,\cH_\mu^c)$'s.

Let $\mu=-\lambda+\rho$ and identify $\cO_0\simeq\cD_\R$. For $f\in C_c^\infty(\cO_0)$ and $z\in \cD$ we define
\[T_\mu f(z) := \int_{\cD_\R} K_\mu (z,x)f(x)\, dx, \]
where $dx$ denotes the Lebesgue measure on the open subset $\cD_\R\subseteq\overline{\fn}$.

\begin{Thm}
$T_\mu$ factors to a unitary isomorphism $\wE\to \cH_\mu$. Furthermore $T_\mu$ is a $\fg^c$-intertwining operator. In particular, the representation of $\fg^c$ on $\cE_+/\cN$ integrates to an irreducible unitary representation of $\widetilde{G}^c$ such that $T_\mu$ is an equivalence of representations.
\end{Thm}

\begin{proof}
Using that $\kappa_\lambda(x,y)=\kappa_\lambda(y,x)=K_\mu (x,y)$ ($x,y\in\cD_\R$) and Lemma \ref{le:4.4} we get for $\mu(iZ_0)\ll0$:
\begin{align*}
 \ip{T_\mu f}{T_\mu h}_\mu^c &=\int_{\cD} T_\mu f(z)\overline{T_\mu h(z)}\,d\nu_\mu(z)\\
 &= \int_{\cD} \int_{\cD_\R }\int_{\cD_\R} K_\mu (z, x)f(x) \overline{K_\mu(z,y)}\overline{h(y)}\, dx\, dy\, d\nu_\mu (z)
\\
 &= \int_{\cD_\R}\int_{\cD_\R}f(x)\overline{h(y)}\int_{\cD}  K_\mu (z,x ) \overline{K_\mu(z,y)}\, d\nu_\mu(z)\, dx\, dy\\
 &= \int_{\cD_\R }\int_{\cD_\R }f(x )\overline{h(y)} \kappa_\lambda (x,y) \, dy\,dx = \ip{f}{h}_\lambda\, .
\end{align*} 
The interchanging of the integrals is allowed because $x$ and $y$ are contained in the compact subsets $\supp f$ and $\supp h$ of $\cD_\R$ and hence the reproducing kernels in the integral are bounded. The step from the third to the fourth equality uses the reproducing property of $K_\mu$. Although for the computation we assumed $\mu(iZ_0)\ll0$, the general case now follows by analytic continuation. The intertwining property follows from the fact that the $\overline N MAN$ decomposition in $G$ is just the restriction
of the $P^c_-K_\C^c P^c_+$ decomposition in $G^c_\C$ and hence $J_\mu(g,x)=j_\lambda(g,x)$ for $h\in H$ and $x\in\cD_\R\subseteq\cD$.
\end{proof}

\subsection{The Intertwining Operator into $L^2_\mu(\widetilde{G}^c/\widetilde{H})$}\label{se:7.4}

For $\mu\in\C$ we denote by $L^2_\mu(\widetilde{G}^c/\widetilde{H})$ the space of measurable functions $u:\widetilde{G}^c/\widetilde{H}\to\C$ such that $Z=Z(\widetilde{G}^c)\subseteq\widetilde{K}^c$ acts on $u$ by $\chi_\mu$ and $|u|\in L^2(\widetilde{G}^c/Z\widetilde{H})$. According to \cite{OO91} there exists a function $F: \cD\times \widetilde{G}^c/\widetilde{H}\to \C$, holomorphic in the first argument, such that the map
\[\Lambda_\mu f(g\widetilde{H})=\int_{\cD} f(z)\overline{F(z,g\widetilde{H})}\, d\nu_\mu (z)\]
defines an intertwining operator $\Lambda_\mu : \cH_\mu^c\to L^2_\mu (\widetilde{G}^c/\widetilde{H})$. Hence, $(\rho_\mu,\cH_\mu^c)$ occurs in $L^2_\mu(\widetilde{G}^c/\widetilde{H})$ as a discrete summand, and it is further shown that it occurs with multiplicity one. Furthermore, for a fixed $g\in\widetilde{G}^c$ the function $z\mapsto F(z,g\widetilde{H})$ is bounded and hence contained in $\cH_\mu$.

For $f\in C^\infty_c(\cD_\R)$ and $g\in\widetilde{G}^c$ define
\[S_\mu f(g\widetilde{H}) :=\int_{\cD_\R } f(x) \overline{F(x,g\widetilde{H})}\, dx \, .\]

\begin{Thm}
We have $S_\mu = \Lambda_\mu \circ T_\mu$. In particular $S_\mu$ extends to an isometric embedding $S_\mu : \wE\to \Im \Lambda_\mu \subset L^2_\mu (\widetilde G^c/\widetilde H)$.
\end{Thm}

\begin{proof}
The proof is a simple change of order of integrals, using that the integral over $\cD_\R$ is only over the compact set $\supp f$:
\begin{align*}
 (\Lambda_\mu\circ T_\mu)f(g\widetilde{H}) &= \int_{\cD} \int_{\cD_\R} f(x)K_\mu(z,x)\overline{F(z,g\widetilde{H})}\, dx\,d\nu_\mu(z)\\
 &= \int_{\cD_\R} f(x) \overline{\int_{\cD} F(z,g\widetilde{H})\overline{K_\mu(z,x)}\, d\nu_\mu(z)} dy \\
 &= \int_{\cD_\R } f(x) \overline{F(x,g\widetilde{H})}\, dx = S_\mu f(g\widetilde{H}) \,,
\end{align*} 
where we have used the reproducing property of $K_\mu$ in the last step.
\end{proof}

\section{The non-Riemannian open $H$-orbits}\label{se:nonR}

We show that on the non-Riemannian open $H$-orbits the Berezin form is only positive semidefinite for $\lambda=\rho$ which constructs the trivial representation of $G^c$ on $\wE=\C$. This is done via a rank two reduction, more precisely we first show that every pair $(\fg,\fh)$ contains either the pair $(\sl(3,\R),\so(1,2))$ or the pair $(\sp(2,\R),\gl(2,\R))$ in a certain way which allows us to use computations for these two particular examples.

\subsection{Rank two examples}

We discuss the rank two examples $(G,H)=(\SL(3,\R),\SO(1,2))$ and $(\Sp(2,\R),\GL(2,\R))$ in detail.

\begin{Ex}[$G/H=\SL(n+1,\R)/\SO (1,n)$]\label{ex:NonReflectionPositive}
Let $G=\SL(n+1,\R)$, $n\geq2$, with involution $\tau(g)=I_{1,n}(g^{-1})^t I_{1,n}$, then $H=\SO(1,n)$ and $G^c=\SU(1,n)$. We choose $X_0=\frac{1}{n+1}\diag(n,-1,\ldots,-1)\in\fg$ and identify $\overline{N}\simeq\R^n$ by
$$ \R^n\to\overline{N}, \quad x\mapsto\overline{n}_x=\left(\begin{array}{cc}1& 0\\x&I_n\end{array}\right). $$
Then on $\overline{N}$ the Berezin kernel is given by
$$ \kappa_\lambda(x, y) = \alpha(\tau(\overline{n}_{-x})\overline{n}_y)^{\lambda- \rho} = |1-\langle x,y\rangle|^{\lambda- \rho}, \qquad x,y\in\R^n, $$
where we identify $\fa_\C^*\simeq\C$ by $\lambda\mapsto\frac{n+1}{n}\lambda(X_0)$ so that $\rho=\frac{n+1}{2}$. There are two open $H$-orbits on $\cB\simeq\R{\rm P}^n$, and their intersections with the open dense Bruhat cell $\overline{N}\cdot b_0\subseteq\cB$ are given by
$$ \cO_0\cap\overline{N} = \{x\in\R^n\mid |x|<1\} = \cD_\R, \qquad \cO_1\cap\overline{N} = \{x\in\R^n\mid |x|>1\}. $$
The restriction of the Berezin kernel $\kappa_\lambda$ to $\{x\in\R^n\mid |x|<1\}$ is positive semidefinite if and only if $\lambda-\rho$ is contained in the Wallach set
$$ \cW = (-\infty,0)\cup\{0\} = (-\infty,0]. $$
We claim that the restriction to the other open orbit $\{x\in\R^n\mid |x|>1\}$ is positive semidefinite if and only if $\lambda= \rho$, i.e. $\kappa_\lambda\equiv1$. In fact, consider the distribution $f_\lambda=\delta_x-\delta_y$ on $\R^n$ with fixed $x,y\in\R^n$, $|x|,|y|>1$. Then $f_\lambda$ corresponds to a distribution $f$ on $\cB$ via the identification \eqref{eq:f-lambda} and by Lemma~\ref{le:4.4} we have
$$ \langle f,f\rangle_\lambda = (|x|^2-1)^{\lambda-\rho}+(|y|^2-1)^{\lambda- \rho}-2|1-\langle x,y\rangle|^{\lambda- \rho}. $$
If now $\lambda- \rho<0$ then for $|x|,|y|\gg1$ with $x\perp y$ we have $\langle f,f\rangle_\lambda<0$.
On the other hand, for $\lambda-\rho>0$ we have $\langle f,f\rangle_\lambda<0$ if $x\perp y$ and $|x|,|y|$ are
close to $1$. Approximating the distributions $\delta_x$ and $\delta_y$ by smooth bump functions, we obtain
that $\langle\cdot,\cdot\rangle_\lambda$ cannot be positive semidefinite if $\lambda+\rho\neq0$.
\end{Ex}

\begin{Ex}[$G/H=\Sp(n,\R)/\GL(n,\R)$]\label{ex:Non2}
Let $G=\Sp(n,\R)$, $n\geq2$, with involution $\tau(g)=I_{n,n}(g^{-1})^t I_{n,n}$, then $H\simeq\GL(n,\R)$ and $G^c\simeq\Sp(n,\R)$. Since $(G,H)$ is a Cayley type symmetric pair, $H$ is conjugate to $G_0=\{\diag(g,(g^{-1})^t)\mid g\in\GL(n,\R)\}\simeq\GL(n,\R)$, more precisely $g_0G_0g_0^{-1}=H$ with
$$ g_0 = g_0^{-1} = \frac{1}{\sqrt{2}}\left(\begin{array}{cc}I_n&I_n\\I_n&-I_n\end{array}\right) \in G. $$
We choose $X_0=\frac{1}{2}\diag(I_n,-I_n)$ and identify $\overline{N}\simeq\Sym(n,\R)$ by
$$ \Sym(n,\R)\to\overline{N}, \quad x\mapsto\overline{n}_x=\left(\begin{array}{cc}I_n& 0\\x&I_n\end{array}\right). $$
Then on $\overline{N}$ the Berezin kernel is given by
$$ \kappa_\lambda( x, y) = \alpha(\tau(\overline{n}_{-x})\overline{n}_y)^{\lambda- \rho} = \det(I_n-xy)^{\lambda- \rho}, \qquad x,y\in\Sym(n,\R), $$
where we identify $\fa_\C^*\simeq\C$ by $\lambda\mapsto\frac{2}{n}\lambda(X_0)$ so that $\rho=\frac{n+1}{2}$.
To describe the $H$-orbits in $G/\Pmax$ we first consider the $G_0$-orbits in $G/\Pmax$. They are all contained in the open dense Bruhat cell $\overline{N}\subseteq G/\Pmax$ and of the form
$$ \widetilde{\cO}_j = \{x\in\Sym(n,\R)\mid \sgn(x)=(j,n-j)\} \qquad (1\leq j\leq n), $$
where $\sgn(x)$ denotes the signature of the quadratic form on $\R^n$ corresponding to $x$. Then the open $H$-orbits are given by $\cO_j=g_0\widetilde{\cO}_j$. To find the intersection of $\cO_j$ with $\overline{N}\simeq\Sym(n,\R)$ we have to write elements of $\cO_j$ in the $\overline{N}G_0N$ decomposition, so we write
\begin{align*}
 g_0\left(\begin{array}{cc}I_n& 0\\x&I_n\end{array}\right) &= \frac{1}{\sqrt{2}}\left(\begin{array}{cc}I_n+x&I_n\\I_n-x&-I_n\end{array}\right)\\
 &= \left(\begin{array}{cc}I_n&0\\y&I_n\end{array}\right)\left(\begin{array}{cc}g& 0\\ 0&(g^{-1})^t\end{array}\right)\left(\begin{array}{cc}I_n&z\\ 0 &I_n\end{array}\right)\\
 &= \left(\begin{array}{cc}g&gz\\yg&ygz+(g^{-1})^t\end{array}\right),
\end{align*}
then $y=(I_n-x)(I_n+x)^{-1}$ and $x=(I_n-y)(I_n+y)^{-1}$. Hence,
$$ \cO_j\cap\overline{N} = \left\{y\in\Sym(n,\R)\, \left|\, \begin{array}{c}\det(I_n+y)\neq0,\\\sgn((I_n-y)(I_n+y)^{-1})=(j,n-j)\end{array}\right. \right\}. $$
Now let us specialize to the case $n=2$, then the orbits $\cO_0$ and $\cO_2$ are Riemannian and the orbit $\cO_1$ is non-Riemannian. We have $x\in\widetilde{\cO}_1$ if and only if $\sgn(x)=(1,1)$ which is equivalent to $\det(x)<0$. Hence, $y\in\cO_1\cap\overline{N}$ if and only if one of the two determinants $\det(I_n\pm y)$ is positive and the other one negative. Write
$$ y = \left(\begin{array}{cc}a&b\\b&c\end{array}\right), $$
then $\det(I_n\pm y)=1\pm(a+c)+ac-b^2$. Hence,
$$ \cO_1\cap\overline{N} = \left\{\left. \left(\begin{array}{cc}a&b\\b&c\end{array}\right)\, \right|\, -|a+c|<1+ac-b^2<|a+c|\right\}. $$
As in Example~\ref{ex:NonReflectionPositive} consider $f_\lambda=\delta_x-\delta_y$ with $x,y\in\cO_1\cap\overline{N}$, then
$$ \langle f,f\rangle_\lambda = |\det(I_n-x^2)|^{\lambda- \rho}+|\det(I_n-y^2)|^{\lambda- \rho} - 2|\det(I_n-xy)|^{\lambda- \rho}. $$
Choosing
$$ x=\left(\begin{array}{cc}s&0\\0&0\end{array}\right) \qquad \mbox{and} \qquad y=\left(\begin{array}{cc}0&0\\0&t\end{array}\right) $$
we have $x,y\in\cO_1\cap\overline{N}$ whenever $|s|,|t|>1$ and
\begin{equation*}
 |\det(I_n-x^2)| = (s^2-1), \qquad |\det(I_n-y^2)| = (t^2-1), \qquad |\det(I_n-xy)| = 1.
\end{equation*}
As in the first example, by choosing $s,t$ either close to $1$ or close to $\infty$ it follows that the Berezin form restricted to the open $H$-orbit $\cO_1$ cannot be positive semidefinite unless $\lambda+\rho=0$.
\end{Ex}

\subsection{Rank two reduction}

We now generalize the above examples to all $H$-orbits which are not Riemannian symmetric spaces. The idea is to reduce to one of the two examples by finding a subalgebra $\fg_i\subseteq\fg$ such that $(\fg_i,\fg_i\cap\fh)\simeq(\sl(3,\R),\so(2,1))$ or $(\sp(2,\R),\gl(2,\R))$. For this we first recall some structure theory.

Recall the strongly orthogonal roots $\alpha_1,\ldots,\alpha_{r}\in\Sigma_+$ from Section~\ref{sec:OpenHorbits} which we order such that $\alpha_{i+1}$ is the maximal root which is strongly orthogonal to $\alpha_1,\ldots,\alpha_i$. Denote by $\fa^{*}_+ $ the span of $\alpha_1,\ldots ,\alpha_r$ and by $\fa_-^*$ its orthogonal complement, then $\fa^*=\fa^*_+\oplus\fa^*_-$. Identifying $\fa\simeq\fa^*$ via the Killing form we also get a decomposition $\fa=\fa_+\oplus\fa_-$ with the properties that $\fa_-=\{H\in\fa\mid (\forall j=1,\ldots ,r)\, \alpha_j(H)= 0\}$ and $\fa_+$ is, via a Cayley transform, isomorphic to the maximal abelian subspace $\fa_\fh$ in $\fh\cap \fp$ of Lemma~\ref{lem:MaxAbelianInPcapH}. We can therefore identify $\alpha_i$ with its restriction to $\fa_+$. Recall also that $\fg_{0,\C}=\fk^c_\C$ to connect our statement with the original statement of Moore which we now recall, see \cite{M64} or \cite[Thm. 2.1]{S84}. (Note that the statement by Moore concerns a full Cartan subalgebra $\mathfrak{t}=\mathfrak{t}_0\oplus \fa$ in $\fg_0$. But the span of the $\alpha_j$ is the same if we use  $i\ft_0^*+\fa^*$ or $\fa^*$ and every root in $\Sigma$ is a restriction of a root in $\Sigma(\fg_\C,\ft_\C)$.)

\begin{Thm}[C. C. Moore]
Let the notation be as above. Then the following holds true:
\begin{enumerate}
\item The set of non-zero restrictions of elements of $\Sigma^+$ to $\fa_+$ is one of the following two sets:
\begin{itemize}
\item[(I)] $\{ \alpha_j, \frac{1}{2}(\alpha_i\pm \alpha_j)\mid 1\le i\le r, 1\le i < j\le r\}$,
\item[(II)] $\{\alpha_i, \frac{1}{2}\alpha_i, \frac{1}{2}(\alpha_i \pm \alpha_j)\mid 1\le i\le r, 1\le i < j\le r\}$.
\end{itemize}
\item The case (I) occurs if and only of $G^c/K^c$ is of tube type.
\item The restrictions of roots in $\Sigma_0^+=\{\alpha\in \Sigma^+\mid \alpha (X_0)=0\}=\Sigma^+ (\fa ,\fg_0)$ to $\fa_+$ are precisely those of the form $\frac{1}{2}(\alpha_i-\alpha_j)$ in case (I) and additionally $\frac{1}{2}\alpha_j$ in case (II). The restrictions of roots in $\Sigma_+$ are precisely those of the form $\frac{1}{2}(\alpha_i+\alpha_j)$ and $\alpha_j$ in case (I) and additionally $\frac{1}{2}\alpha_j$ in case (II).
\item Let $\beta\in\Sigma_0^+$. If $\beta|_{\fa_+}=0$, then $\beta $ is strongly orthogonal to all $\alpha_j$. If $\beta|_{\fa_+} = \frac{1}{2}\alpha_i$, then $\beta$ is strongly orthogonal
to all $\alpha_j$, $j\not= i$.
If $\beta|_{\fa_+} = \frac{1}{2}(\alpha_i-\alpha_j)$, ($i<j$), then $\beta$ is strongly orthogonal to all $\alpha_k$, $k\not=i,j$; moreover, $\beta +\alpha_j$ is not a root. 
\item The roots $\alpha_1, \ldots ,\alpha_{r}$ are all long roots. In case (II) only one root length occurs in $\Sigma$.
\item Unless $\fa_-= \{0\}$ the strongly orthogonal roots $\alpha_1,\ldots ,\alpha_r$ are the only restricted roots of multiplicity one.
\end{enumerate}
\end{Thm}

From these structural results it is easy to identify the open $H$-orbits in $\cB$ which are Riemannian symmetric spaces:

\begin{Cor}
The open $H$-orbit $\cO_i=H\cdot\tilde{w}_jb_0$ is a Riemannian symmetric space if and only if $\tilde{w}_i^2$ is contained in the center of $H$. This is precisely the case if either $i=0$, or if $i=r$ and $G^c/K^c$ is of tube type.
\end{Cor}

Now consider a non-Riemannian open $H$-orbit $\cO_i$ in $\cB$ ($1\leq i\leq r$). The following lemma constructs a subalgebra $\fg_i\subseteq\fg$ such that $(\fg_i,\fg_i\cap\fh)$ is either isomorphic to $(\sl(3,\R),\so(1,2))$ or $(\sp(2,\R),\gl(2,\R))$, and such that the (unique) non-Riemannian open $(H\cap G_i)$-orbit in $G_i/(\Pmax\cap G_i)$ embeds into the non-Riemannian open $H$-orbit $\cO_i$ in $G/\Pmax$, where $G_i$ is the analytic subgroup of $G$ with Lie algebra $\fg_i$.

\begin{Lemma}\label{lem:SubalgebrasSL3Sp2}
\begin{enumerate}
\item Let $1\leq i<r$, then there exists a $\tau$-stable subalgebra $\fg_i\subseteq\fg$ such that either $(\fg_i,\fg_i\cap\fh)\simeq(\sp(2,\R),\gl(2,\R))$ or $(\fg_i,\fg_i\cap\fh)\simeq(\sl(3,\R),\so(1,2))$. Moreover, $\fg_i$ commutes with $s_j'$ for all $1\leq j<i$, and $s_i'$ acts on $\fg_i$ as in Example~\ref{ex:NonReflectionPositive} or \ref{ex:Non2}.
\item Let $1\leq i\leq r$ and assume that $\fg^c$ is not of tube type, then there exists a $\tau$-stable subalgebra $\fg_i\subseteq\fg$ such that $(\fg_i,\fg_i\cap\fh)\simeq(\sl(3,\R),\so(1,2))$ and $\fg_i$ commutes with $s_j'$ for all $1\leq j<i$, and $s_i'$ acts on $\fg_i$ as in Example~\ref{ex:NonReflectionPositive}.
\end{enumerate}
\end{Lemma}

\begin{proof}
We first assume $1\leq i<r$ and put $\alpha=\alpha_i$. Let $\beta\in\Sigma_+$ be a root whose restriction to $\fa_+$ is equal to $\frac{1}{2}(\alpha_i+\alpha_{i+1})$ and consider the root string $\alpha-n\beta$. By restricting to $\fa_+$ it is clear that $\alpha-n\beta$ is a root for $n=0$, but not for $n=-1$ and $n=3$. Let $n$ be maximal such that $\alpha-n\beta$ is a root, then $\frac{n}{2}|\beta|^2=\langle\alpha,\beta\rangle=\frac{1}{2}|\alpha|^2>0$ so that $n=1$ or $n=2$. In particular $\gamma=\alpha-\beta$ is a root. We treat the two cases $n=1$ and $n=2$ separately:
\begin{enumerate}
\item If $n=1$ then $\alpha-2\beta$ is not a root and $|\beta|=|\alpha|$. Consider the root string $\beta-m(\beta-\alpha)$, then this is a root for $m=0$ and $m=1$, but not for $m=-1$. Hence, $\frac{m}{2}|\beta-\alpha|^2=\langle\beta,\beta-\alpha\rangle=|\beta|^2-\langle\alpha,\beta\rangle=\frac{1}{2}|\alpha|^2$. Since $|\beta-\alpha|^2=|\beta|^2-2\langle\alpha,\beta\rangle+|\alpha|^2=|\alpha|^2$ it follows that $m=1$ so that $2\alpha-\beta=\beta-2(\beta-\alpha)$ is not a root. Hence, the roots of the rank two subalgebra generated by $\fg_{\pm\alpha}$ and $\fg_{\pm\beta}$ are $\pm\alpha$, $\pm\beta$ and $\pm\gamma$. Choose non-trivial elements $X_\beta\in\fg_\beta$, $X_\gamma\in\fg_\gamma$ and put $X_\alpha=[X_\beta,X_\gamma]\in\fg_\alpha$, $X_{-\alpha}=\theta X_\alpha$, $X_{-\beta}=\theta X_\beta$, $X_{-\gamma}=\theta X_\gamma$, then $X_{\pm\alpha}$, $X_{\pm\beta}$ and $X_{\pm\gamma}$ generate an $8$-dimensional subalgebra $\fg_i$ of $\fg$ isomorphic to $\sl(3,\R)$. Since $\tau=-\theta$ on $\fn\oplus\overline{\fn}$ and $\tau=\theta$ on $\fg_0$ it follows that $\fg_i$ is $\tau$-stable and $\fg_i\cap\fh=\fg_i^\tau\simeq\so(1,2)$.
\item If $n=2$ then $\alpha-2\beta$ is a root and $|\alpha|^2=2|\beta|^2$. Since $|\frac{1}{2}(\alpha_i+\alpha_{i+1})|^2=\frac{1}{2}|\alpha|^2$ it is clear that $\beta|_{\fa_-}=0$ so that $\beta=\frac{1}{2}(\alpha_i+\alpha_{i+1})$. This implies that $\beta$, $\gamma$ and $\beta\pm\gamma$ are the only positive roots of the rank two subalgebra generated by $\fg_{\pm\alpha}$ and $\fg_{\pm\beta}$. As above one constructs a $10$-dimensional $\tau$-stable subalgebra $\fg_i$ of $\fg$ such that $(\fg_i,\fg_i\cap\fh)\simeq(\sp(2,\R),\gl(2,\R))$.
\end{enumerate}
It remains to show that the constructed subalgebras $\fg_i$ commute with $s_j'$ for $1\leq j<i$, and that $s_i'$ acts in the given way. The first statement follows from Moore's Theorem: all roots constructed above are strongly orthogonal to $\alpha_j$ for $1\leq j<i$ and $s_j'\in\exp(\fg_{\alpha_j}+\fg_{-\alpha_j})$. Moreover, $s_i'=\exp(\frac{\pi}{2}(E_i-F_i))$ with $E_i\in\fg_{\alpha_i}$ and $F_i\in\fg_{-\alpha_i}$, and the statement now follows from the explicit isomorphism $(\fg_i,\fg_i\cap\fh)\simeq(\sp(2,\R),\gl(2,\R))$ resp. $(\fg_i,\fg_i\cap\fh)\simeq(\sl(3,\R),\so(1,2))$.

To show the second statement we may assume $i=r$ and $\fg^c$ not of tube type. Then for $\alpha=\alpha_r$ we can choose a root $\beta\in\Sigma_+$ whose restriction to $\fa_+$ is equal to $\frac{1}{2}\alpha_r$. Similar arguments as above show that the the roots $\alpha$, $\beta$ and $\gamma=\alpha-\beta$ construct a subalgebra $\fg_i$ isomorphic to $\sl(3,\R)$. The rest of the proof is analogous to the first part.
\end{proof}

Combining Lemma~\ref{lem:SubalgebrasSL3Sp2} with Example~\ref{ex:NonReflectionPositive} and \ref{ex:Non2} now shows:

\begin{Thm}\label{thm:8.6}
Let $\mathcal{O}_i$ ($1\leq i\leq r$) be a non-Riemannian open $H$-orbit in $G/\Pmax$. Then the Berezin form $\langle\cdot,\cdot\rangle_\lambda$ restricted to $C_c^\infty(\mathcal{O}_i)$ is positive semidefinite if and only if $\lambda-\rho=0$. In this case the construction in Section~\ref{sec:ReflectionPositivity} yields the trivial representation of $G^c$ on the one-dimensional Hilbert space $\wE=\C$.
\end{Thm}

\section{The Hardy--Littlewood--Sobolev inequality}\label{sec:HLS}
In this final section we give a short overview of the application of reflection
positivity to the Hardy--Littlewood--Sobolev inequality, a very basic result in analysis on
Euclidian space and on the sphere. Several proofs have been given, often
involving rearrangement inequalities; and a crucial part of the HLS
inequality was the optimal constant found in 1983 by E. Lieb~\cite{L83}. In a recent paper by R. Frank and E. Lieb~\cite{FL10} one finds a
new proof of certain cases of the sharp HLS inequality, using in an essential
way reflection positivity of inversions in hyperplanes and spheres (see also \cite{FL11}). It is
a remarkable aspect of reflection positivity, whose origin was completely
different, and with very natural interpretations in representation theory, that
it also may lead to HLS. We shall here briefly indicate how the argument goes,
and of course one may speculate about similar applications of the many
generalizations of reflection positivity that we have discussed in this paper.

Consider the Hermitian form
$$I_{\lambda}[f,h] = \int_{\mathbb R^n}\int_{\mathbb R^n}
\frac{f(x)\overline{h(y)}}{|x-y|^{\lambda}}
\,dx\,dy$$
and recall the HLS inequality relating this with the $L^p$ norm
$$|I_{\lambda}[f,h]| \leq C_{n, \lambda, p} ||f||_p ||h||_p$$
where $2/p + \lambda/n = 2,$ and the optimal constant is
$$ C_{n, \lambda, p} = \pi^{\lambda/2}\frac{\Gamma((n - \lambda)/2)}{\Gamma(n - \lambda/2)}
\left(\frac{\Gamma(n)}{\Gamma(n/2)}\right) ^{1 - \lambda/n}$$
with explicit optimizers. This holds true for $0 < \lambda < n$ in general, and the 
reflection positivity will give it for $n = 1, 2$ and for $n - 2 \leq \lambda < n$ for
$n \geq 3.$ The argument uses the well-known conformal invariance of $I_{\lambda}$
and the observation that, in the indicated range,
$$I_{\lambda}[\Theta_H f, f] \geq 0$$
for $f \in L^p$ with support in a closed half-space determined by a hyperplane $H$; here
$\Theta_H$ denotes the reflection in this hyperplane. The conformal invariance means,
that one may also consider reflection in spheres (where the action then also contains a
factor of a suitable power of the Jacobian) and that there is a similar inequality for
reflections in spheres; it also means, that using stereographic projection (which is
conformal) the HLS inequality also holds on the $n$-sphere, and here the optimizer
is simply the constant function (and its images under the conformal group).

Now the argument goes roughly as follows: For an $L^p$-function $f(x)$, let $f^i(x)$ be equal
to $f(x)$ on one side of a hyperplane (or inside a ball) and even with respect to the
reflection $\Theta_H$ (or the ball reflection); similarly let $f^o(x)$ equal $f(x)$ on
the other side of the hyperplane (or outside the ball) and even. Then
$$\frac{1}{2}\left(I_{\lambda}[f^i] + I_{\lambda}[f^o]\right) \geq I_{\lambda}[f]$$
and the inequality is strict unless $f$ is even (with respect to the reflection in question).

Then an additional result about finite, non-negative measures, invariant under suitably
many reflections in hyperplanes and spheres, says that these are absolutely continuous
with respect to Lebesgue measure, and the density is $(1 + |x|^2)^{-n}$ (or translates).

Assume now that $f$ is an optimizer in HLS, and that $f^i$ and $f^o$ both have the same
$L^p$-norm as $f$; hence $f$ is even, and the measure $f^p\,dx$ satisfies the assumptions
about invariant measures -- leading to the desired form of the optimizer. Some additional
arguments are needed in the case of $\lambda = n-2$, but it is remarkable how this
proof of Frank and Lieb is using reflection positivity in a simple way.

\begin{Rem}
We remark that for $G=\SO(n+1,1)$ the Hermitian form $I_\lambda[\cdot,\cdot]$ is precisely the complementary series inner product in the non-compact realization of the principal series representation $\pi_\lambda$ on $\overline{N}\simeq\R^n$. The optimizer of the Hardy--Littlewood--Sobolev inequality is the $K$-spherical vector of $\pi_\lambda$. We further note that the complementary series representation $\pi_\lambda$ extends to a representation on $L^p(\R^n)$ by isometric operators.
\end{Rem}


\begin{thebibliography}{22}

\bibitem[B75]{B75} F. A. Berezin: Quantization in complex symmetric spaces. {\it Izv. Akad. Nauk SSSR Ser. Mat.} \textbf{39} (1975), no.~2, 363--402, 472.

\bibitem[B{\'O}{\O}96]{BOO96} T. Branson, G. \'Olafsson, and B. {\O}rated:
Spectrum generating operators and intertwining operators for representations induced from a
maximal parabolic subgroup. \textit{J. Funct. Anal. } \textbf{135} (1996), 163--205.

\bibitem[vDH97a]{vDH97a}  G. van Dijk and S. C. Hille: Canonical representations related to hyperbolic spaces. {\it J. Funct. Anal.} \textbf{147} (1997), 109--139.

\bibitem[vDH97b]{vDH97b} G. van Dijk and S. C. Hille:  Maximal degenerate representations, Berezin kernels and canonical representations. In: 
Komrakov, B., Krasil'shchik, J., Litvinov, G., Sossinky, A. (eds.) \textit{Lie groups and Lie algebras, their representations, generalizations, and applications.} Dordrecht: Kluwer Academic, 1997, pp. 1--15.

\bibitem[vDM98]{vDM98} 
G. van Dijk and  V. F. Molchanov: The Berezin form for rank one para-Hermitian symmetric spaces.  {\it  J. Math. Pures Appl.  } \textbf{77} (1998), no.~8, 747--799.

\bibitem[vDM99]{vDM99} 
G. van Dijk and V. F. Molchanov:  Tensor products of maximal degenerate series representations of the group
${\rm SL} (n,\mathbb R)$.   {\it J. Math. Pures Appl. }  \textbf{78}  (1999),  no.~1, 99--119.

\bibitem[vDP99]{vDP99} 
G. van Dijk and A. Pasquale: Canonical representations of ${\rm Sp}(1,n)$ 
associated with representations of ${\rm Sp}(1)$.  \textit{Comm. Math. Phys.} \textbf{202}  (1999),  651--667.

\bibitem[E83]{E83} T. J. Enright: Unitary representations for two real forms of a semi simple Lie algebra: A theory of comparison. In: {\it  Lie Group Representations, I (College Park, Maryland, 1982/1983)}, LNM \textbf{1024}, Springer 1983, pp. 1--29.

\bibitem[FL10]{FL10}
R. Frank and E. Lieb: Inversion positivity and the sharp Hardy--Littlewood--Sobolev inequality. \textit{Calc. Var. Partial Differential Equations} \textbf{39} (2010), no.~1-2, 85--99.

\bibitem[FL11]{FL11} R. Frank and E. Lieb: Spherical reflection positivity and the Hardy--Littlewood--Sobolev inequality. In: {\it Concentration, functional inequalities and isoperimetry (Providence, RI, 2011)}, Contemp. Math. \textbf{545}, Amer. Math. Soc. 2011, pp. 89--102.

\bibitem[FP05]{FP05} J. Faraut and M. Pevzner:
Berezin kernels and analysis on Makarevich spaces. 
\textit{Indag. Math. (N.S.)} \textbf{16} (2005), no.~3-4, 461--486. 

\bibitem[H78]{H78} S. Helgason:  \textit{Differential Geometry, Lie groups, and
Symmetric Spaces}, Academic Press, 1978.

\bibitem[H00]{H00}
S. Helgason: \textit{Groups and Geometric Analysis},
A.M.S., Providence, RI, 2000.

\bibitem[HO97]{HO97}  J. Hilgert and G. \'Olafsson: \textit{Causal Symmetric Spaces. Geometry and Harmonic Analysis.}
Perspectives in Mathematics, 18. Academic Press, Inc., San Diego, CA, 1997.

\bibitem[H99]{H99}  S. C. Hille: \textit{Canonical Representations}, Ph.D. Thesis, Leiden University, 1999.
 
 \bibitem[HS97]{HS97} E. Hulett and C. U. S\'anchez: An algebraic characterization of $R$-spaces.
 \textit{Geometriae Dedicata} \textbf{67} (1997), 349--365.
  
\bibitem[JZ17]{JZ17} A. Jaffe and Z. Liu: Planar Para Algebras, Reflection Positivity.
\textit{Comm. Math. Phys.}  \textbf{352} (2017),  95-- 133.

\bibitem[JJ16]{JJ16} A. Jaffe and B. Janssens:   Characterization of reflection positivity: Majoranas and spins.
\textit{Comm. Math. Phys.}  \textbf{346} (2016),  1021--1050.  

\bibitem[JJ17]{JJ17} A.  Jaffe and B. Janssens: Reflection positive doubles.
\textit{J. Funct. Anal. } \textbf{272} (2017),  3506--3557. 

 \bibitem[JP15a]{JP15a} A. Jaffe and F. L. Pedrocchi:  Reflection positivity for parafermions.
 \textit{Comm. Math. Phys.}  \textbf{337} (2015),  455--472. 
 
 \bibitem[JP15b]{JF15b} A. Jaffe and F. L. Pedrocchi: Reflection positivity for Majoranas.
 \textit{Ann. Henri Poincar\'e} \textbf{16}  (2015), 189--203. 
 
\bibitem[J86]{J86}   P. E. T. Jorgensen:  Analytic continuation of local representations of Lie groups,
\textit{Pacific J. Math.} \textbf{125} (1986), 397--408. 

\bibitem[J87]{J87} P. E. T. Jorgensen: Analytic continuation of local representations of symmetric spaces,
\textit{J. Funct. Anal.} \textbf{70} (1987), 304--322.  
 
\bibitem[J{\'O}98]{JO98}  P. E. T.  Jorgensen and G. \'Olafsson:  Unitary Representations of Lie Groups with
Reflection Symmetry. {\it J. Funct. Anal.} {\bf 158} (1998), 26--88.

\bibitem[J{\'O}00]{JO00}  P. E. T.  Jorgensen and G. \'Olafsson:  Unitary representations and Osterwalder-Schrader
Duality. In: Ed. R.
S. Doran, V. S. Varadarajan:  {\it The Mathematical Legacy of Harish-Chandra: A
Celebration of Representation Theory and Harmonic Analysis}, PSPM, AM, 2000.

\bibitem[K85]{K85} S. Kaneyuki: On classification of parahermitian symmetric spaces. \textit{Tokyo J. Math.} \textbf{8} (1985), 473--482. 

\bibitem[K87]{K87} S. Kaneyuki: On orbit structure of compactifications of paraphermitian symmetric spaces, \textit{J. J. Math.} \textbf{13} (1987), 333--370.

\bibitem[K00]{K00} S. Kaneyuki: Graded Lie algebras, related geometric structures and pseudo-hermitian symmetric spaces. In: 
J. Faraut et al., \textit{Analysis and geometry on complex homogeneous domains}. Progress in Mathematics, 185. Birkh\"{a}user Boston, Inc., Boston, MA, 2000. 
 
\bibitem[KA88]{KA88} S. Kaneyuki and H. Asano: Graded Lie algebras and generalized Jordan triple systems. \textit{Nagoya Math. J.} \textbf{112} (1988), 81--115.


\bibitem[KL83]{KL83} A. Klein and L. Landau: From the Euclidean group to the
Poincar\'e group via Osterwalder--Schrader positivity, \textit{Comm. Math. Phys.}  {\bf 87} (1983),
469--484

\bibitem[KW65a]{KW65a} A. Kor\'anyi, and J. A. Wolf: Realization of Hermitian symmetric spaces as generalized half-planes.
\textit{Ann. Math.} \textbf{81} (1965), 265--288. 

\bibitem[KW65b]{KW65b} A. Kor\'anyi, and J. A. Wolf: Generalized Cayley transformations of
bounded symmetric domains. \textit{Amer. J. Math.} \textbf{87} (1965),  899--939.

\bibitem[L83]{L83} E. Lieb: Sharp constants in the Hardy--Littlewood--Sobolev and related inequalities. \textit{Ann. of Math. (2)} \textbf{118} (1983), no.~2, 349--374.

\bibitem[Lo77]{Lo77} O. Loos: Jordan triple systems, $R$-spaces, and bounded symmetric domains.
\textit{Bull. Amer. Math. Soc.} \textbf{77} (1971), 558--561. 

\bibitem[LM75]{LM75} M. L\"uscher and G. Mack: Global conformal invariance and quantum field theory, 
\textit{Comm. Math. Phys.} \textbf{41} (1975), 203--223.

\bibitem[M82]{M82} 
T. Matsuki: Orbits on affine symmetric spaces under the action of parabolic subgroups, \textit{Hiroshima Math. J.}. \textbf{12} (1982), no.~2, 307--320.

\bibitem[MS14]{MS14} J. M{\"o}llers and B. Schwarz:  Structure of the degenerate principal series on symmetric $R$-spaces and small representations.
\textit{J. Funct. Anal.} \textbf{266} (2014), 3508--3542.

\bibitem[M64]{M64} C. C. Moore: Compactification of symmetric spaces II: Cartan Domains. \textit{Amer. J. of Math.} \textbf{86} (1964), 358--378.

\bibitem[N65]{N65} 
T. Nagano: Transformation groups on compact symmetric spaces. \textit{Trans. Amer. Math. Soc.} \textbf{118} (1965), 428--453.

\bibitem[N{\'O}14]{NO14} K.-H. Neeb  and G. \'Olafsson: Reflection Positivity and Conformal Symmetry, 
\textit{J. Funct. Anal.} \textbf{266} (2014),  2174--2224

\bibitem[NO15a]{NO15a} K.-H. Neeb and G. \'Olafsson:  Reflection 
positive one-parameter groups and dilations, {\it Complex 
Analysis and Operator Theory} {\bf 9:3} (2015), 653--721 

\bibitem[NO15b]{NO15b} K.-H. Neeb and G. \'Olafsson:
Reflection positivity for the circle group. In ``Proceedings of the 30th International Colloquium on Group 
Theoretical Methods,'' Journal of Physics: Conference Series 
{\bf 597} (2015), 012004; arXiv:math.RT.1411.2439 

\bibitem[N\'O17a]{NO16} K.-H Neeb,  G. \'Olafsson:  {\it 
Reflection positivity on spheres and hyperboloids}, in preparation

\bibitem[N\'O17b]{NO17b} K.-H Neeb,  G. \'Olafsson: \textit{Reflection Positivity - A Representation Theoretic Perspective}. In preparation.

\bibitem[N{\'O}00]{NO00} A. Neumann and G. {\'O}lafsson: Minimal and Maximal Semigroups Related to Causal Symmetric
Spaces, \textit{Semigroup Forum} \textbf{61} (2000) 57--85.

\bibitem[{\'O}87]{O87} G. \'Olafsson:
Fourier and Poisson transformation associated to a semisimple
symmetric space. {\it Invent. Math.} {\bf 90} (1987), 605--629.

\bibitem[{\'O}91]{O91} G. \'Olafsson: Symmetric spaces of hermitian type. \textit{Diff. Geo. and App.} \textbf{1} (1991), 195--233.

\bibitem[{\'O}00]{O00} G. \'Olafsson: Analytic Continuation in Representation Theory and Harmonic Analysis. In: \textit{Global Analysis and Harmonic Analysis}, ed. J. P. Bourguignon, T. Branson, and O. Hijazi.  Seminares et Congr, vol 4, (2000), 201--233.
Pub.: The French Math. Soc

\bibitem[\'OP12]{OP12}
G. \'Olafsson and A. Pasquale: The ${\rm Cos}^\lambda$ and ${\rm Sin}^\lambda$ transforms as intertwining operators between generalized principal series representations of ${\rm SL}(n+1,\mathbb K)$, \textit{Adv. Math.} \textbf{229} (2012), 267--293.

\bibitem[\'OPR13]{OPR13} G. \'Olafsson, A. Pasquale and B. Rubin:  Analytic and Group-Theoretic Aspects of the Cosine Transform. In: Eds: E. T. Quinto, F. Gonzalez, J. Christensen, Geometric
Analysis and Integral Geometry,  \textit{Contemporary Math.} \textbf{598}, 167--188,  Amer. Math. Soc., Providence, RI, 2013. 

\bibitem[{\'O}{\O}91]{OO91}
G. {\'O}lafsson and B. {\O}rsted: The holomorphic discrete series of an affine
symmetric space and representations with reproducing kernels, {\it Trans.
Amer. Math. Soc.} {\bf 326} (1991), 385--405.

\bibitem[{\'O}{\O}96]{OO96} G. \'Olafsson and B. {\O}rsted: Generalizations of the Bargmann transform. In: Lie theory and
its applications in physics (Clausthal, 1995) (H.-D. Doebner, V. K. Dobrev and J. Hilgert,
eds.), World Scientific, River Edge, New Jersey, 1996.

\bibitem[{\O}Z95]{OZ95}
B. {\O}rsted and G. Zhang: Generalized principal series representations and tube domains, \textit{Duke Math. J.} \textbf{78} (1995), 335--357.

\bibitem[OS73]{OS73} K. Osterwalder and R. Schrader: Axioms for Euclidean Green's functions, \textit{Comm. Math. Phys.} {\bf 31} (1973), 83--112.

\bibitem[OS75]{OS75} K. Osterwalder and R. Schrader: Axioms for Euclidean Green's functions.~II, {\it Comm. Math. Phys.} \textbf{42} (1975), 281--305.

\bibitem[R79]{R79} W. Rossmann: The structure of semisimple symmetric spaces, \textit{Canad. J. Math.} \textbf{31} (1979), 157--180.

\bibitem[R13]{R13} B. Rubin: Funk, cosine, and sine transforms on Stiefel and Grassmann manifolds. {\it J. Geom. Anal. } {\bf 23} (2013),  1441--1497.

\bibitem[S93]{S93}
S. Sahi:  Unitary representations on the Shilov boundary of a symmetric tube domain. In:  Representation Theory of Groups and Algebras (Providence, RI, 1993), \textit{Contemp. Math.} \textbf{145}, 275--286, Amer. Math. Soc., Providence, RI, 2013.

\bibitem[S95]{S95}
S. Sahi: Jordan algebras and degenerate principal series, \textit{J. Reine Angew. Math.} \textbf{462}
(1995), 1--18.

\bibitem[S84]{S84} H. Schlichtkrull:  One-dimensional $K$-types  in finite dimensional representations of semisimple Lie groups: A generalization of Helgason's theorem.
 {\it  Math. Scand. } \textbf{54} (1984), 279--294.

\bibitem[S86]{S86}   R. Schrader: Reflection positivity for the complementary series of
$\SL (2n, \C)$, \textit{Publ. Res. Inst. Math. Sci.} \textbf{22} (1986), 119--141. 

\bibitem[T79]{T79} M. Takeuchi: On conjugate loci and cut loci of compact symmetric spaces. {II}.
\textit{Tsukuba J. Math.} \textbf{3} (1979), no.~1, 1--29.

\bibitem[T87]{T87} M. Takeuchi: Basic transformations of symmetric $R$-spaces.
\textit{Osaka J. Math.} \textbf{25} (1988), 259--297.

\bibitem[VR76]{VR76}
M. Vergne and H. Rossi: Analytic continuation of the holomorphic discrete series of a semi-simple Lie group, \textit{Acta Math.} \textbf{136} (1976), no.~1-2, 1--59.

\bibitem[VW90]{VW90}
D. A. Vogan and N. R. Wallach: Intertwining operators for real reductive groups, \textit{Adv. Math.} \textbf{82} (1990), no.~2, 203--243.

\bibitem[W79]{W79}
N. R. Wallach: The analytic continuation of the discrete series. I, II, \textit{Trans. Amer. Math. Soc.} \textbf{251} (1979), 1--17, 19--37.

\bibitem[W72]{W72} J. A. Wolf: The fine structure on hermitian symmetric spaces. In: W. Boothy and G. Weiss, Eds., \textit{Symmetric Spaces},
Marcel Dekker, New York, 1972.

\bibitem[Z95]{Z95}
G. Zhang: Jordan algebras and generalized principal series representations, \textit{Math. Ann.} \textbf{302} (1995), 773--786.

\end{thebibliography}
\end{document}